\documentclass[11pt,fleqn]{article}
\usepackage{amsmath,amssymb,graphicx}
\usepackage{mathrsfs}
\usepackage{bold-extra}
\usepackage{algorithmic}
\usepackage[plain, section]{algorithm}
\usepackage{lscape}
\usepackage[pass]{geometry}
\usepackage{longtable}
\usepackage{tikz}
\usepackage{array}
\usepackage{color}
\usepackage{url}

\newcolumntype{E}{>{\setbox0=\hbox\bgroup}c<{\egroup}@{}}

\definecolor{grey}{rgb}{0.5,0.5,0.5}

\usetikzlibrary{arrows,patterns,decorations.pathreplacing,decorations.pathmorphing,backgrounds,positioning,fit}
\begin{document}

\sloppy
\newtheorem{axiom}{Axiom}
\newtheorem{conjecture}[axiom]{Conjecture}
\newtheorem{corollary}[axiom]{Corollary}
\newtheorem{definition}[axiom]{Definition}
\newtheorem{example}[axiom]{Example}
\newtheorem{lemma}[axiom]{Lemma}
\newtheorem{observation}[axiom]{Observation}
\newtheorem{proposition}[axiom]{Proposition}
\newtheorem{theorem}[axiom]{Theorem}

\newcommand{\proof}{\emph{Proof.}\ \ }
\newcommand{\qed}{~~$\Box$}
\newcommand{\rz}{{\mathbb{R}}}
\newcommand{\nz}{{\mathbb{N}}}
\newcommand{\zz}{{\mathbb{Z}}}
\newcommand{\eps}{\varepsilon}
\newcommand{\cei}[1]{\lceil #1\rceil}
\newcommand{\flo}[1]{\left\lfloor #1\right\rfloor}
\newcommand{\seq}[1]{\langle #1\rangle}

\renewcommand{\algorithmicrequire}{\textbf{Input:}}
\renewcommand{\algorithmicensure}{\textbf{Output:}}

\newcommand{\toSolve}{\fontencoding{OT1}\fontfamily{cmtt}\fontseries{m}\fontshape{sc}\fontsize{12}{16pt}\selectfont}

\newcommand{\qap}{\mbox{\sc QAP}}
\newcommand{\minqap}{\mbox{$\min$-{\sc QAP}}}
\newcommand{\maxqap}{\mbox{$\max$-{\sc QAP}}}
\newcommand{\onion}{\mbox{\sc Onion}}
\newcommand{\ocone}{\mbox{\sc OnionCone}}
\newcommand{\aaa}{\alpha}
\newcommand{\minctv}{\mbox{$\min$-{\sc CTV}}}
\newcommand{\maxctv}{\mbox{$\max$-{\sc CTV}}}
\newcommand{\cbar}{\overline{C}{}}
\newcommand{\var}{\mbox{\sc Var}}

\title{{\bf Heuristics for the data arrangement problem on regular trees}}
\author{\sc Eranda \c{C}ela\thanks{{\tt cela@opt.math.tu-graz.ac.at}. 
Institut f\"ur Optimierung und Diskrete Mathematik, TU Graz, Steyrergasse 30, A-8010 Graz, Austria}
\and
\sc Rostislav Stan\v{e}k\thanks{{\tt rostislav.stanek@uni-graz.at}.
Institut f\"ur Statistik und Operations Research, Universit\"at Graz,
 Universit\"atsstra{\ss}e 15,  Bauteil E/III, A-8010 Graz, Austria} 
}
\date{}
\maketitle

\begin{abstract}

The data arrangement problem on regular trees (DAPT) consists in assigning the vertices of a given graph $G$ to the leaves of a $d$-regular tree $T$ such that 
the sum of the pairwise distances of all pairs of leaves in $T$ which correspond to edges of $G$ is minimised. 
Luczak and Noble~\cite{LuNo02} have shown that this problem is  $NP$-hard for every fixed $d\ge 2$. 

In this paper we propose construction and local search heuristics for the DAPT and introduce a lower bound for this problem. 
The analysis of the performance of the heuristics is based on two considerations: a) the quality of the solutions produced by the heuristics as  compared to  
the respective lower bounds b) for a special class of instances  with  known optimal solution we evaluate the gap between the optimal value of the objective function 
 and the objective function value attained by the  heuristic solution, respectively.

\medskip\noindent\emph{Keywords.}
Combinatorial optimisation; data  arrangement  problem; regular trees; 
heuristics.
\end{abstract}

\medskip
\section{Introduction}
\label{intro:sec}

Given an undirected graph $G = (V(G), E(G))$ with $|V(G)| = n$,  an undirected graph $H=(V(H),E(H))$ with $|V(H)|\ge n$ and some subset $B$  of the vertex set   of 
 $H$, $B\subseteq V(H)$, with $|B|\ge n$, the {\sl generic graph embedding problem}\/(GEP) consists of finding an injective  embedding of the vertices of 
$G$ into the vertices in $B$   such that some prespecified objective function is minimised. 
Throughout this paper  we will call $G$ {\sl the  guest graph}\/ and $H$ {\sl the host graph}. 
A commonly used objective function maps an embedding $\phi\colon V(G)\to B$ to $$\sum_{(i,j)\in E(G)}d(\phi(i)\phi(j))\, ,$$ where $d(x,y)$ denotes 
 the length of the shortest  path between $x$ and $y$ in $H$. 
The host graph $H$ may be a weighted or a non-weighted graph; in the second cases the path lengths coincide with the respective number of edges.  
Given a non-negative number $A\in \rz$,  the decision version of the $GEP$ asks whether there is an injective embedding $\phi\colon V(G)\to B$ such that 
the objective function does not exceed $A$.  
\smallskip

Different versions of GEP have been studied in the literature; the {\sl linear arrangement problem}, where the 
guest graph is a one dimensional equidistant grid with $n$ vertices, see \cite{Chu84,JuMo92,Shi79}, is probably  the most prominent among them.
A number of other classical and well known  combinatorial optimisation problems can be seen as special cases of the GEP, 
as  e.g.\ the Hamiltonian cycle problem, the Hamiltonian path problem and  the graph isomorphism problem 
(see e.g.~\cite{CeSta13} for a more detailed discussion of the relationship between these problems).  
\smallskip

This paper deals with the version of the GEP where the guest graph $G$ has $n$ vertices, the host graph $H$ is a complete $d$-regular tree of  height  $\lceil\log_d{n}\rceil$ 
and the set $B$ consists of the leaves of $H$. From now on we we will denote the host graph by $T$. The height of $T$ as specified above guarantees 
that the number $|B|$ of leaves  fulfills $|B|\ge n$ 
and that the number of the predecessors of the leaves in $T$ is smaller than $n$.  Thus $\lceil\log_d{n}\rceil$ is the smallest height of a $d$-regular tree which is able to accommodate 
an injective embedding  of the vertices of  the guest graph on its leaves.   This problem is originally motivated by real problems in communication 
systems and was first posed by  Luczak and Noble~\cite{LuNo02}. We will call this version of the GEP {\sl the data arrangement problem on regular trees (DAPT)}.  
Luczak and Noble~\cite{LuNo02} have shown that the DAPT  is $NP$-hard for every fixed $d\ge 2$. 
The question about the computational complexity of the DAPT in the case where the guest graph is a tree, posed by  Luczak and Noble  in \cite{LuNo02}, is still open.  
 In this perspective the development of heuristic approaches to efficiently find good solutions to DAPT  is   a natural task. 
There are plenty of heuristics for different versions of the GEP in the literature, especially for the linear arrangement problem, see  e.g.\ the papers by 
Petit~\cite{Pet98,Pet03} for nice and comprehensive reviews. However, to our knowledge there are no specific heuristic approaches to solve the DAPT 
and no benchmark instances have been developed for this problem yet. 
In this paper we make a first step in this direction and  propose construction and local search approaches as well as a lower bound for the DAPT, much in the spirit of 
\cite{Pet98,Pet03} which deal with the linear arrangement problem. 
In order to evaluate  the performance of the proposed heuristics  we generate a number of families  of test instances 
some of them being polynomially solvable or having a known optimal objective function value.

 The paper is organised as follows. 
Section~\ref{general:sec} discusses some  general properties of the problem  and introduces the notation used throughout the paper. 
In Section~\ref{lowbou:sec} we derive a lower bound for optimal objective function value to be used  in the evaluation of the performance of solution heuristics. 
Section~\ref{heurist:sec} introduces the proposed {\em heuristics}. Sections~\ref{TestInst:sec}, \ref{numeric:sec}\/ and \ref{sec:conclusionsAndOutlook} discuss the  test instances, the numerical results and some  conclusions and outlook, respectively.

\section{Notations and general properties of the DAPT} 
\label{general:sec}
Consider a guest graph $G=(V,E)$ with $n$ vertices, $|V|=n$, and a host graph $T$ which is a $d$-regular tree  of height $h$, $h:=\lceil\log_d{n}\rceil$. 
Let $B$ be the set of leaves of $T$. Notice that due to the above choice of $h$ we get the following upper bound for the number $b=|B|$ of leaves:
\begin{equation}\label{nrleaves:eq}
b:=|B|=d^h=d^ {h-1} d < nd \,  .
\end{equation}

\begin{definition} An {\em arrangement}\/ is an injective  mapping $\phi: V \to B$. 
The {\em  data arrangement problem on regular trees (DAPT)}\/ asks for  an {\em arrangement} $\phi$ that minimises the {\em objective value}  $OV(G, d, \phi)$
\begin{equation}
	\label{eq:definition}
	OV(G, d, \phi):=\sum_{(u, v) \in E}{d_T(\phi(u), \phi(v))}\, ,
\end{equation}
where $d_T(\phi(u), \phi(v))$ denotes the length of the $\phi(u)$-$\phi(v)$-path in the $d$-regular tree  $T$. Such an arrangement is called an {\em optimal arrangement}.  
An  instance of the DAPT is fully determined by the guest graph and the parameter $d$ of the regular tree $T$ which serves as host graph. 
Such an instance of the problem will be denoted by $DAPT(G,d)$.  
\end{definition}

Figure~\ref{fig:graph} shows a guest graph $G$ with vertices $\{v_1,v_2,v_3,v_4,v_5\}$ 
and Figure~\ref{fig:minimumArrangement} shows  a $3$-regular tree of height $2=\lceil\log_3{5}\rceil$  as a host graph
  together with a minimum arrangement. The numbers in the leaves of $T$ denote the vertex indices mapped to the  leaves, respectively,  
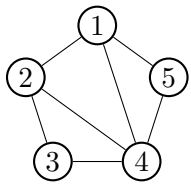
\begin{figure}[htbp]
	\begin{center}
		\begin{tikzpicture}[
				scale=1,
				node/.style={circle, draw=black!100, fill=white!0, thick, inner sep=0pt, minimum size=5mm}
			]
		
			\node[node] (node1) at (0.00, +1.00) {1};
			\node[node] (node2) at (-0.95, +0.31) {2};
			\node[node] (node3) at (-0.59, -0.81) {3};
			\node[node] (node4) at (+0.59, -0.81) {4};
			\node[node] (node5) at (+0.95, +0.31) {5};
	
			\draw [-] (node1) to (node2);
			\draw [-] (node2) to (node3);
			\draw [-] (node3) to (node4);
			\draw [-] (node4) to (node5);
			\draw [-] (node5) to (node1);

			\draw [-] (node1) to (node4);
			\draw [-] (node2) to (node4);
		\end{tikzpicture}
	\end{center}
	\caption{\em A guest graph.}
	\label{fig:graph}
\end{figure}
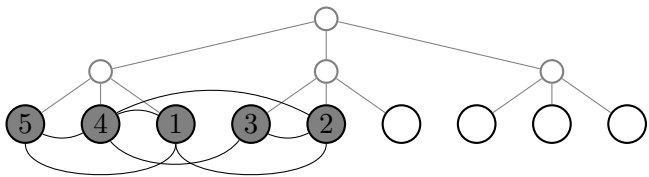
\begin{figure}[htbp]
	\begin{center}
		\begin{tikzpicture}[
				xscale=1, yscale=0.7,
				leafUsed/.style={circle, draw=black!100, fill=black!50, thick, inner sep=0pt, minimum size=5mm},
				leafUnused/.style={circle, draw=black!100, fill=white!100, thick, inner sep=0pt, minimum size=5mm},
				otherNode/.style={circle, draw=black!50, fill=white!100, thick, inner sep=0pt, minimum size=3mm}
			]
			\node[otherNode] (node0) at (4.00, 2.00) {};

			\node[otherNode] (node1) at (1.00, 1.00) {};
			\node[otherNode] (node2) at (4.00, 1.00) {};
			\node[otherNode] (node3) at (7.00, 1.00) {};

			\node[leafUsed] (node4) at (0.00, 0.00) {5};
			\node[leafUsed] (node5) at (1.00, 0.00) {4};
			\node[leafUsed] (node6) at (2.00, 0.00) {1};
			\node[leafUsed] (node7) at (3.00, 0.00) {3};
			\node[leafUsed] (node8) at (4.00, 0.00) {2};
			\node[leafUnused] (node9) at (5.00, 0.00) {};
			\node[leafUnused] (node10) at (6.00, 0.00) {};
			\node[leafUnused] (node11) at (7.00, 0.00) {};
			\node[leafUnused] (node12) at (8.00, 0.00) {};

			\draw [-, black!50] (node0) to (node1);
			\draw [-, black!50] (node0) to (node2);
			\draw [-, black!50] (node0) to (node3);

			\draw [-, black!50] (node1) to (node4);
			\draw [-, black!50] (node1) to (node5);
			\draw [-, black!50] (node1) to (node6);

			\draw [-, black!50] (node2) to (node7);
			\draw [-, black!50] (node2) to (node8);
			\draw [-, black!50] (node2) to (node9);
				
			\draw [-, black!50] (node3) to (node10);
			\draw [-, black!50] (node3) to (node11);
			\draw [-, black!50] (node3) to (node12);

			\draw [-] (node6) to [out=-90,in=-90] (node8);
			\draw [-] (node8) to [out=-150,in=-30] (node7);
			\draw [-] (node7) to [out=-120,in=-60] (node5);
			\draw [-] (node5) to [out=-150,in=-30] (node4);
			\draw [-] (node4) to [out=-90,in=-90] (node6);

			\draw [-] (node6) to [out=150,in=30] (node5);
			\draw [-] (node8) to [out=145,in=35] (node5);
		\end{tikzpicture}
	\end{center}
	\caption{\em An optimal  arrangement which objective value is $20$.}
	\label{fig:minimumArrangement}
\end{figure}

\noindent {\bf Notations.} From now on let the set of vertices of the guest  graph $G$ be given as  $V(G) = \{v_1, \dots v_n\}$ and let  $m:=|E|$ be its number of edges. 
We denote the {\em set of neighbours}\/ of any vertex $v$ by $\Gamma(v)$. We denote by $h(T)$ the height of a ($d$-)regular tree $T$.
A {\sl basic subtree} $T^\prime$ of the $d$-regular tree $T$ is a $d$-regular subtree of $T$ with $h(T^\prime)=h(T)-1$ rooted at some  son of the root of  $T$.   
For every $d, h\in \nz$, $2 \leq d \leq n$, the leaves of a $d$-regular tree  of height $h$  are denoted by $b_1,b_2,\ldots,b_{d^h}$ such that $b_{(i - 1) d^ {h-1}+1}, b_{(i - 1) d^ {h-1}+2},\ldots, b_{i \cdot d^ {h-1}}$ are the leaves of the $i$-th basic subtree. 
This order of the leaves is called {\sl the canonical order}.
If the leaves are labelled according to the canonical ordering then the pairwise distances between the leaves of a $d$-regular tree are given by a simple formula. 
\begin{observation}\label{distance:obse}
Let $T$ be  a $d$-regular tree of height $h$ and let its leaves be labelled according to the canonical ordering. 
The  distances between the leaves in $T$ are given as  $d_T(b_t,b_j)=2l$, where 
\[ l:=\min\left \{k\in \{1,2,\ldots,h\}\colon \left  \lfloor \frac{t-1}{d^k}\right \rfloor = \left  \lfloor \frac{j-1}{d^k}\right \rfloor \right \}\, ,\]
  for all leaves $b_t,b_j$ of $T$ with $t,j \in \{1,2,\ldots,d^h\}$.
\end{observation} 
\begin{proof}
First let us observe that for all $t,j\in \{1,2,\ldots,d^h\}$, $\lfloor \frac{t-1}{d^l}\rfloor = \lfloor \frac{j-1}{d^l}\rfloor$ implies $\lfloor \frac{t-1}{d^{l+1}}\rfloor = \lfloor \frac{j-1}{d^{l+1}}\rfloor$, for all $l\in \{1,2,\ldots,h-1\}$.

We prove the claim  by induction on $h$. If $h=1$ then $T$ has $d$ leaves labelled by $1,2,\ldots,d$,  their pairwise distances are all equal to  $2$
and $\lfloor\frac{t-1}{d}\rfloor = \lfloor\frac{j-1}{d}\rfloor=0$, so the claim holds. 
Assume that the claim holds for regular trees of height up to $h-1$. Consider now a tree of height $h$ with leaves labelled by $b_1,b_2,\ldots, b_{d^h}$  in the canonical ordering and let $b_t$, $b_j$ be two leaves of it. 
\smallskip

Let $t=(i_t-1)d^{h-1}+r_t$ and $j=(i_j-1)d^{h-1}+r_j$ with $i_t,i_j \in  \{1,2,\ldots,d\}$ and $r_t,r_j\in \{1,2,\ldots, d^{h-1}\}$.
Clearly $\lceil \frac{t-1}{d^{h-1}}\rceil=i_t-1$ and  $\lceil \frac{j-1}{d^{h-1}}\rceil=i_j-1$. Thus, 
if $l:=\min\{k\in \{1,2,\ldots,h\} \colon \lfloor \frac{t-1}{d^k}\rfloor =\lfloor\frac{j-1}{d^k}\rfloor \}=h$,  if and only if  $i_t\neq i_j$, or equivalently, 
$b_t$, $b_j$ are leaves  of  different basic subtrees of $T$. For leaves of different basic subtrees we have  $d_T(b_i,b_j)=2h$ and hence,  the claim holds in this case.  
\smallskip

Otherwise $l:=\min\{k\in \{1,2,\ldots,h\}\colon  \lfloor \frac{t-1}{d^k}\rfloor =\lfloor\frac{j-1}{d^k}\rfloor\}\le h-1$ which implies $\lfloor \frac{t-1}{d^{h-1}}\rfloor=\lfloor \frac{j-1}{d^{h-1}}\rfloor$. Thus  $b_t$ and $b_j$ 
are leaves  of  the same basic subtree of $T$. 
Let this be the $r$-th basic  subtree $T_r$ of $T$ with leaves $b_{(r-1)d^{h-1}+s}$ with $s=1,2,\ldots, d^{h-1}$. 
In the canonical ordering in $T_r$ these leaves would be labelled by $b_s$,  for $s=1,2, \ldots,d^{h-1}$. 
Let $t=(r-1)d^{h-1}+s_t$ and $j=(r-1)d^{h-1}+s_j$ for $s_t,s_j \in \{1,2, \ldots,d^{h-1}\}$.
$T_r$ is d-regular tree of height $h-1$ and hence $d_{T_r}(b_{s_t},b_{s_j})=2l$ holds,  
where   $l:=\min\{k\in \{1,2,\ldots,h-1\} \colon \lfloor \frac{s_t-1}{d^k}\rfloor = \lfloor\frac{s_j-1}{d^k}\rfloor \}$,  according to our inductive assumption. 
Finally notice that $d_{T_r}(b_{s_t},b_{s_j})=d_T(b_t,b_j)$  and 
\begin{eqnarray*}
 l:&=& \min\left \{k\in \{1,2,\ldots,h-1\} \colon \left \lfloor \frac{s_t-1}{d^k}\right \rfloor = \left \lfloor\frac{s_j-1}{d^k}\right  
\rfloor\right \}\\
&=& 
\min\left \{k\in \{1,2,\ldots,h-1\} \colon 
 \left \lfloor \frac{(r-1)d^{h-1}+s_t-1}{d^k}\right \rfloor = \left  \lfloor\frac{(r-1)d^{h-1}+s_j-1}{d^k}\right \rfloor \right \} 
\end{eqnarray*}
hold.  This completes the proof. 
\qed
\end{proof}
\smallskip 

\begin{definition}
For a given  arrangement $\phi$ let  $B_u = \{\phi(1), \dots, \phi(n)\}$ be called  the {\em set of used leaves}. 
If $B_u = \{b_i, \dots b_{i + n - 1}\}$ holds for some $1 \leq i \leq b - n + 1$,  $\phi$ is called  a {\em contiguous arrangement}. 
\end{definition}

Let us notice that not every instance of the DAPT possesses necessarily   a contiguous optimal  arrangement as illustrated by the following example.

\begin{example}\label{nocontig:opt}
A DAPT instance which does not possess  any contiguous optimal arrangement. 

The  guest graph $G$ with $12$ nodes  is represented in     
Figure~\ref{fig:noMinimumContinuousArrangementExistsGraph}. Consider  $d=4$. 
The optimal arrangement $\phi$ represented in 
Figure~\ref{fig:noMinimumContinuousArrangementExistsMinimumArrangement} is not contiguous. 
In both pictures  we identify the vertices with their indices, thus we write $i$ instead of $v_i$, $i=1,2,\ldots,12$ for simplicity. The optimal value $OV(G,4,\phi)$ equals  $28$ and can be written as $OV(G,4,\phi)=4*a(\phi)+2(m-a(\phi))$, where $m=11$ is the number of edges of the guest graph and 
$a(\phi)=3$ is the number of edges of $G$ with end-vertices mapped by $\phi$ into different basic subtrees of $T$. 

We show now that for every contiguous arrangement $\psi$,  $a(\psi)>3$ holds, implying that $OV(G,4,\psi)>OV(G,4,\phi)$.  
In order to see  that we make a case distinction according to the number of neighbours of vertex $v_1$ embedded together with $v_1$ 
in the same basis subtree. 
Assume this  number is $1$ and w.l.o.g.\ vertex $v_2$ is mapped together with $v_1$ to the leaves of the same basis subtree, say $T_1$.  
Then of course $v_4$, $v_7$ and $v_1$ are not  mapped by $\psi$ into leaves of $T_1$. So $a(\psi)\ge 3$. 
Moreover, due to the contiguity of $\psi$ for at  least one of the paths $\{v_4,v_5,v_6\}$, $\{v_7,v_8,v_9\}$, $\{v_{10},v_{11},v_{12}\}$ holds 
that not all of its vertices are mapped into the leaves of a common basic subtree. 
Due to that there is definitely one more edge (not incident to vertex $v_1$) whose end-vertices are mapped by $\psi$ into leaves of different 
basic subtrees, and hence $a(\psi)\ge 4$. 
The other cases where the number of neighbours of $v_1$ mapped together with $v_1$ into the leaves of the same basic subtree is $2$ or $3$ can be argued upon analogously\footnote{In fact we can show that the DAPT is polynomially solvable in the case that the guest graph is an 
{\sl extended star}\/ as in this example and for some suitable choices of $d$. In this case  the optimal arrangement has a particular structure and 
is in  general not contiguous. This and other polynomially solvable special cases of the DAPT are discussed in another paper we are working in.}.
\end{example}

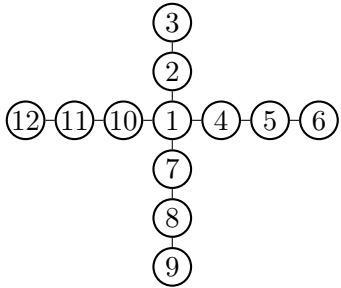
\begin{figure}[htbp]
	\begin{center}
		\begin{tikzpicture}[
				scale=1.3,
				node/.style={circle, draw=black!100, fill=white!0, thick, inner sep=0pt, minimum size=5mm}
			]
				
			\node[node] (node1) at (0.00, 0.00) {1};
			\node[node] (node2) at (0.00, 0.50) {2};
			\node[node] (node3) at (0.00, 1.00) {3};
			\node[node] (node4) at (0.50, 0.00) {4};
			\node[node] (node5) at (1.00, 0.00) {5};
			\node[node] (node6) at (1.50, 0.00) {6};
			\node[node] (node7) at (0.00, -0.50) {7};
			\node[node] (node8) at (0.00, -1.00) {8};
			\node[node] (node9) at (0.00, -1.50) {9};
			\node[node] (node10) at (-0.50, 0.00) {10};
			\node[node] (node11) at (-1.00, 0.00) {11};
			\node[node] (node12) at (-1.50, 0.00) {12};
			
			\draw [-] (node1) to (node2);
			\draw [-] (node2) to (node3);
			\draw [-] (node1) to (node4);
			\draw [-] (node4) to (node5);
			\draw [-] (node5) to (node6);
			\draw [-] (node1) to (node7);
			\draw [-] (node7) to (node8);
			\draw [-] (node8) to (node9);
			\draw [-] (node1) to (node10);
			\draw [-] (node10) to (node11);
			\draw [-] (node11) to (node12);
		\end{tikzpicture}
	\end{center}
	\caption{\em The guest graph of the DAPT instance in Example~\ref{nocontig:opt}.}
	\label{fig:noMinimumContinuousArrangementExistsGraph}
\end{figure}

\begin{figure}[htbp]
	\begin{center}
		\begin{tikzpicture}[
				xscale=1.1, yscale=0.7,
				leafUsed/.style={circle, draw=black!100, fill=black!50, thick, inner sep=0pt, minimum size=5mm},
				leafUnused/.style={circle, draw=black!100, fill=white!100, thick, inner sep=0pt, minimum size=5mm},
				otherNode/.style={circle, draw=black!50, fill=white!100, thick, inner sep=0pt, minimum size=3mm}
			]
			\node[otherNode] (node0) at (3.75, 2.00) {};

			\node[otherNode] (node1) at (0.75, 1.00) {};
			\node[otherNode] (node2) at (2.75, 1.00) {};
			\node[otherNode] (node3) at (4.75, 1.00) {};
			\node[otherNode] (node4) at (6.75, 1.00) {};

			\node[leafUsed] (node5) at (0.00, 0.00) {1};
			\node[leafUsed] (node6) at (0.50, 0.00) {2};
			\node[leafUsed] (node7) at (1.00, 0.00) {3};
			\node[leafUnused] (node8) at (1.50, 0.00) {};
			\node[leafUsed] (node9) at (2.00, 0.00) {4};
			\node[leafUsed] (node10) at (2.50, 0.00) {5};
			\node[leafUsed] (node11) at (3.00, 0.00) {6};
			\node[leafUnused] (node12) at (3.50, 0.00) {};
			\node[leafUsed] (node13) at (4.00, 0.00) {7};
			\node[leafUsed] (node14) at (4.50, 0.00) {8};
			\node[leafUsed] (node15) at (5.00, 0.00) {9};
			\node[leafUnused] (node16) at (5.50, 0.00) {};
			\node[leafUsed] (node17) at (6.00, 0.00) {10};
			\node[leafUsed] (node18) at (6.50, 0.00) {11};
			\node[leafUsed] (node19) at (7.00, 0.00) {12};
			\node[leafUnused] (node20) at (7.50, 0.00) {};

			\draw [-, black!50] (node0) to (node1);
			\draw [-, black!50] (node0) to (node2);
			\draw [-, black!50] (node0) to (node3);
			\draw [-, black!50] (node0) to (node4);

			\draw [-, black!50] (node1) to (node5);
			\draw [-, black!50] (node1) to (node6);
			\draw [-, black!50] (node1) to (node7);
			\draw [-, black!50] (node1) to (node8);

			\draw [-, black!50] (node2) to (node9);
			\draw [-, black!50] (node2) to (node10);
			\draw [-, black!50] (node2) to (node11);
			\draw [-, black!50] (node2) to (node12);

			\draw [-, black!50] (node3) to (node13);
			\draw [-, black!50] (node3) to (node14);
			\draw [-, black!50] (node3) to (node15);
			\draw [-, black!50] (node3) to (node16);

			\draw [-, black!50] (node4) to (node17);
			\draw [-, black!50] (node4) to (node18);
			\draw [-, black!50] (node4) to (node19);
			\draw [-, black!50] (node4) to (node20);

			\draw [-] (node5) to [out=-55,in=-125] (node6);
			\draw [-] (node6) to [out=-55,in=-125] (node7);
			\draw [-] (node5) to [out=-60,in=-120] (node9);
			\draw [-] (node9) to [out=-55,in=-125] (node10);
			\draw [-] (node10) to [out=-55,in=-125] (node11);
			\draw [-] (node5) to [out=-62,in=-118] (node13);
			\draw [-] (node13) to [out=-55,in=-125] (node14);
			\draw [-] (node14) to [out=-55,in=-125] (node15);
			\draw [-] (node5) to [out=-64,in=-116] (node17);
			\draw [-] (node17) to [out=-55,in=-125] (node18);
			\draw [-] (node18) to [out=-55,in=-125] (node19);
		\end{tikzpicture}
	\end{center}
	\caption{\em The non-contiguous optimal arrangement which objective value is $28$ for the DAPT instance in Example~\ref{nocontig:opt}.}
	\label{fig:noMinimumContinuousArrangementExistsMinimumArrangement}
\end{figure}
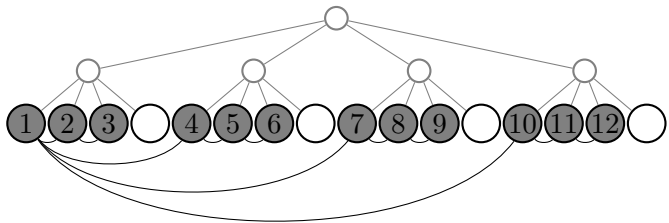

\section{A lower bound}\label{lowbou:sec}

 In a  $DAPT(G,d)$ with vertex set $V(G)$ of size $n$, $n:=|V(G)|$, we have $b:= d^h$ leaves,  where $h=\lceil\log_d{n}\rceil$ is the height of the regular tree. Thus there are $\frac{b!}{(b - n)!}$ possible arrangements and the complete enumeration becomes inefficient  even for very small instances. Further  let us notice that  $2 m \leq OV(G, d, \phi) \leq 2 h m$ holds for every {\em arrangement} $\phi$, where $m$ is the number of edges of  the guest graph $G$.   These bounds are due to the fact that the distance between any two leaves in a regular tree of height $h$ 
 is between $2$ and $2h$. 
\smallskip

Next we introduce the so-called degree lower bound for the DAPT which will be also used   to evaluate the performance of the heuristics 
introduced in this paper. We adapt an idea  used  by {\sc Petit} in~\cite{Pet98}  for the linear arrangement problem. The idea is the construction of locally optimal arrangements for every vertex $v$ of $G$, i.e.\ the construction of  an optimal arrangement of $v$ and its neighbours. Than the contribution of vertex $v$ to the objective function value of any feasible  solution cannot be larger than the objective function value of this locally optimal arrangement divided by $2$.

More precisely, for every $v\in V(G)$ we define a new graph $G'_v = (V'_v, E'_v)$ with the {\em vertex set} $V'_v := V$ and the {\em edge set} $E'_v = \left\{\{v, u\}\colon u \in \Gamma(v)\right \}$.  Thus $G'_v$ is a subgraph of $G$ containing all vertices of $G$ and just the edges incident to $v$.  Obviously, $G'_v$ is the union of a {\em star}\/ and some isolated vertices. An  {\em optimal  arrangement} $\phi_v$  for $DAPT(G^{\prime}_v,d)$ is obtained by  placing $v$  on some  {\em leaf}, say  $b_1$ w.l.o.g.\, 
 and the other neighbours on the {\em leaves} $b_2, \dots, b_{1 + |\Gamma(v)|}$ one by one, where the canonical order of the leaves is adopted. The other vertices of $G$ are  arranged arbitrarily on the remaining leaves $b_{2 + |\Gamma(v)|},\ldots, b_{d^ h}$. Let $OV_v$ denote the objective function value of the above mentioned arrangement for every $v \in V$. It is obvious that $DB(G,d)$ given as below is a lower bound for $DAPT(G,d)$, that is
\begin{equation}
	\label{eq:dg1}
	DB(G, d) = \frac{1}{2} \sum_{v \in V}{OV_v} \leq OV(G, d, \phi) \mbox{ for all arrangements $\phi$.}
\end{equation}
This bound $DG(G, d)$ is called the  {\sl degree bound}.

$DB(G,d)$ can be easily computed because $OV_v$ can be easily computed, given $d$ and the number $|\Gamma(v)|$ of neighbours, for all $v\in V(G)$. 
\begin{lemma}
	\label{lemma:dg}
	Let $G = (V, E)$ be a star graph with $n$ vertices and $2 \leq d \leq n$ a natural number. The optimal  value $OV$ of $DAPT(G,d)$ is given as 
	\begin{equation}
		\label{eq:dg2}
		OV = 2 \left(h \; n - \frac{d^h - 1}{d - 1}\right)\, ,  \mbox{where  $h = \lceil \log_d{n} \rceil$ is the height of the host $d$-regular tree.}
	\end{equation}
	
	\begin{proof}
Let $v:=v_1$ be the central vertex of $G$ with vertex set $\{v_1,v_2,\ldots,v_n\}$. 	It is clear that the optimal arrangement places 
the vertices $v_1,v_2,\ldots,v_n$ into the leaves $b_1$,$b_2$,\ldots, $b_n$ of the $d$-regular tree of height $h$, respectively, where  the leaves are given in the canonical order. 
Consider a partition of the set of leaves into  sets $B_j = \{\mbox{$b$ is a leaf } \colon  d_T(b_1, b)=2j\}$ with $j = 0, \dots, h$. 
It is clear that $B_0=\{b_1\}$, $B_1=\{b_2,\ldots,b_d\}$, and hence $|B_0|=1$, $|B_1|=d-1$.

Generally, for $j=0,1,\ldots,h$, a $d$-regular tree of height $h$ contains  $d^{h-j}$ $d$-regular subtrees of height $j$. Clearly one of these subtrees, say $T_1$ contains $b_1$. This subtree has in turn $d$ $d$-regular subtrees of height $j-1$ and (only) one of those contains $b_1$. 
 The set  $B_j$ consists exactly of the leaves of those $d$-regular subtrees of height $j-1$ of $T_1$ which do  not contain $b_1$. There are clearly $d - 1$ such subtrees with $d^{j-1}$ leaves each. Hence  $|B_j|=(d - 1) d^{j-1}$ for all $j=1,2,\ldots,h$.

Due to $h=\lceil \log_d n \rceil$ we have $d^{h-1} < n\le d^ h$   and hence the leaves of the  basic subtree which contains $b_1$ (and thus hosts  $v_1$) are all occupied.  Consequently  the other basic subtrees have  exactly $n-d^{h-1}>0$ occupied leaves.
Thus we get 
\[ OV=\sum_{j=1}^{h-1} 2j|B_{j}| + 2h(n-d^ {h-1}) =2 (d - 1) \sum_{j=1}^{h-1}jd^{j-1}+2h(n-d^{h-1}).\] 
 Using $\sum_{j = 1}^{h - 1}{d^{j - 1} j} = \frac{\big((d - 1) h - d\big)d^{h - 1} + 1}{\left(d - 1\right)^2}$ we get the lemma.
\qed
	\end{proof}
\end{lemma}

By applying Lemma~\ref{lemma:dg}   to evaluate $OV_v$ in  (\ref{eq:dg1}) as the optimal objective function value of the DAPT with  a guest graph being  star 
graph with $|\Gamma(v)+1|$ vertices  we get:
\begin{theorem}
	\label{theorem:dg}
	Let $G = (V, E)$ be a graph and $2 \leq d \leq n$ a degree of the arrangement tree. Then the degree bound is given as 
	\begin{equation}
		\label{eq:dg3}
		DB(G, d) = \sum_{v \in V} \left( p(v)(|\Gamma(v)| + 1) - \frac{d^{p(v)} - 1}{d - 1}\right)
	\end{equation}
	where
	\begin{equation}
		\label{eq:dg4}
		p(v):= \lceil \log_d{(|\Gamma(v)| + 1)} \rceil \, .
	\end{equation}
\end{theorem}

\section{Heuristic approaches  for  the DAPT}
\label{heurist:sec}
In this section we will  introduce some simple greedy heuristics, a construction heuristic and two local search heuristics for the DAPT. 

\subsection{Simple greedy  approaches}\label{sub:greedy}

A simple greedy strategy considers the leaves  of the guest graph in the canonical order. The first leaf is occupied by a vertex selected at random. 
Then we consider the next leaf in the canonical order, place at it  ``best possible vertex'',   and
repeat this process until all vertices of the guest graph have been placed to some leaf. 
``The best possible vertex'' means here  a vertex  which leads to the {\sl biggest increase}\/ in the objective function value  of the DAPT. We call this 
heuristic {\sc G2}. {\sc G2} is a leaf-driven heuristic. 
Clearly there are also  vertex-driven greedy algorithms which investigate the vertices in some prespecified order and place the current vertex 
to the ``best possible free leaf''. Since   the vertex-driven greedy heuristics we have tested were  outperformed by the leaf-driven greedy heuristic described above we do  not present them in details in this paper. 

 The time complexity of  {\sc G2} is $O(\max\{(m+n)n, n^2 \log{n}\})$. To see  this consider first a pre-processing step to compute the   distances between all pairs of leaves of the arrangement tree  in $O(n^2\log{n})$ time according to 
Observation~\ref{distance:obse}. Then 
$n$ iterations are performed to arrange  the vertices one at a time.
 The computation of  the increase in the objective function value resulting by placing a specific vertex $v$ onto the current leaf  takes 
 $O(|\Gamma(v)|)$ time per each vertex and hence $O(m)$ time for all candidate vertices. Selecting the best among all candidate  vertices  takes another $O(n)$ time.
Thus we obtain a time complexity of $O(n+m)$ per iteration which results to $O((n+m)n)$  for all iterations and to an  overall time complexity of $O(\max\{(m+n)n,n^2\log{n}\})$ (including the pre-processing step).
\smallskip

We have also tested two  very simple search heuristics {\sc BFSG} and {\sc DFSG} which order the vertices of the guest graph according to  
   breadth-first search  or    depth-first search, respectively, after starting at some prespecified vertex. 
Then the vertices are places onto the leaves in the canonical order, i.e.\ the $i$-the vertex according the resulting ordering is placed at the $i$-th leaf,  $i=1,2,\ldots,n$. 

Of course there are a number of  variants of this algorithm. 
We distinguish different implementations for connected and non-connected graphs. 
In  the case of a  connected guest graph $G$ there is a flexibility in choosing  
the starting vertex for  search  algorithm in $G$.
 Depending on the graph structure  the  vertex with the highest degree can be chosen.
 Or  the algorithm is run  for each vertex as starting vertex and then  the best obtained solution is chosen.

In the case of non-connected graphs  we have  to fix the order of the connected components before running the search algorithm for each of them. 
This can be done in many ways, e.g.\ by considering the connected components  in decreasing  order of magnitude.

Clearly,   the  worst-case time complexity depends on the particular implementation in each case. 
In the case of connected graphs we  obtain an $O(n^3)$ algorithm, if the ``best'' starting vertex among all is chosen. 
In the case of non-connected graphs we obtain the same time complexity,  if we choose the best starting vertex in each component 
by running the algorithm as many times as the number of vertices for each component.  

\subsection{A construction heuristic}\label{sub:constrDescr}
Let us now consider the objective function of the problem from another point of view.Let $a_i$, $1\le i\le h$, be the number of edges of the guest graph $G$ whose endpoints  are mapped into leaves of $T$ at a distance $2i$ in the host graph.  

We can state obviously
\begin{equation}
	\label{equation:ch}
	OV(G, d, \phi) = 2 h a_h + 2 (h - 1) a_{h - 1} + \dots + 2 a_1\, ,
\end{equation}
where $a_h + a_{h - 1} + \dots + a_1 = m$ and $m$ is the number of edges of the guest graph $G$. 

 Since our aim is to minimise the  objective value $OV(G, d, \phi)$, we  try first to minimise the coefficient $a_h$ by partitioning  the vertex set $V$ in at most $d$  subsets $V_i$, $1 \leq i \leq d$,  with  $0 \leq \left|V_i\right| \leq \frac{|B|}{d}$.  Then each $V_i$, $1\le i\le d$, is embedded into the leaves of the corresponding basic subtree, which means that the inequalities $(i - 1)d +1 \leq \phi(v) \leq i \; d$ hold for any $v\in V_i$, $1\le i\le d$. Among all  arrangements of this kind we choose one which minimises 
$a_h=\left|\left\{(u, v) \in E | u \in V_i, v \in V_j, i \neq j\right\}\right|$. 
 Then the subproblems $DAPT(G[V_i],d)$, $1\le i\le d$,  (where $G[V_i]$ is the subgraph of $G$ induced by the set of vertices $V_i$) are solved 
in order to determine an arrangement of $V_i$, $1\le i\le d$, into the leaves of the corresponding basic subtree.   
\smallskip

The problem of partitioning  $V$ as described above is strongly related to  the so called 
{\sl  minimum cut  problem with bounded set size (MCBSSP)}\/
described in next subsection. In Subsection~\ref{subsub:heuristic} we present  an approach to solve the $DAPT(G,d)$ by using the idea described 
above and  a heuristic for MCBSSP. 
\subsubsection{A related problem (MCBSSP) and some heuristic approaches }\label{subsub:MCBSSP}
\begin{quote}
{\bf The Minimum Cut  Problem with Bounded Set Size (MCBSSP)}\\
Input: A graph $G=(V,E)$ with $n=|V|$ and two integers $l$, $u$ with $0 < l \le u < n$.\\
Output: A set $X\subset V$ with $l\le |X|\le u$ such that the cut\\ $\delta(X):= \left\{(u, v) \in E | u \in X, v \notin X \right\}$ has minimum cardinality.
\end{quote}

MCBSSP is equivalent to the  so-called  {\sl $(k,n-k)$ cut problem ($k$-$(n-k)$CP)}, investigated by  Feige, Krauthgamer and Nissim~\cite{Feige01}.

\begin{quote}
{\bf The $(k,n-k)$ cut problem ($k$-$(n-k)$CP)}\\
Input: A graph $G=(V,E)$ with $n=|V|$ and an  integer $k$, with $k < n$.\\
Output: A partition of $V$ in  $X$, $Y$ with $ |X|= k$,  $|Y|=n- k$ such that the cut\\ $\delta(X):= \left\{(u, v) \in E | u \in X, v \in Y \right\}$ has minimum cardinality.
\end{quote}  

\noindent Indeed  the equivalence between MCBSSP and $k$-$(n-k)$CP  is trivial: an optimal solution of    MCBSSP in a graph $G$ with input parameters $l,u$  can be obtained by solving $O(n)$ instances of  $k$-$(n-k)$CP  in the same graph $G$ with input parameter $k= u, u+1, \ldots, l$. 
 On the other hand $k$-$(n-k)$CP is just a special case of MCBSSP, when $u=l$ holds. $k$-$(n-k)$CP  is NP-hard  for general $k$ 
as mentioned in Feige et al.~\cite{Feige01}, a special case of it is the {\sl minimum bisection problem}, 
see Garey and Johnson~\cite{Garey1979}.  Thus MCBSSP is also NP-hard for general $l$ and $u$  and there is no hope to optimally solve it in polynomial time (unless $P=NP$). 
\smallskip

We have considered two heuristic approaches  to solve MCBSSP. 
These will  then be applied recursively to obtain a heuristic for the $DAPT(G,d)$ as described above.  

The  first approach is based  on 
a polynomial time approximation algorithm for $k$-$(n-k)$CP with an approximation ratio $O(\log^2n)$
 proposed by  Feige, Krauthgamer and Nissim~\cite{Feige01}.
 (Their algorithm reaches an even better approximation rate for the cases  $k=O(\log n)$
and  $k=\Omega(\log n)$). 
 So in order to obtain a solution of   MCBSSP  in the graph $G$ with parameters $l$ and $u$ we apply the approach of  
Feige et al~\cite{Feige01} to  $k$-$(n-k)$CP in $G$ with  parameter $k$ varying between $l$  and $ u$ and then choose a minimum  cut  among the 
$l-u+1$ obtained solutions  of  $k$-$(n-k)$CP. Since $u-l\le n$ we get a polynomial time approach  for MCBSSP. 
\smallskip

Our second approach for   MCBSSP  makes use of  a simple {\em local search idea}. 
Assume that $l = u$. 
We randomly  partition $V$ in $X$ and $V\setminus X$, where $\emptyset \subset X \subset V$ and $|X| = l = u$. 
We try to decrease the  cardinality of the cut $|\delta(X)|$ by the following pair-exchange approach. 
Consider an other {\em cut} $\delta((X \setminus \{u\}) \cup \{v\})$ for each pair $(u, v)$, where $u \in X$ and $v \notin X$. 
Replace $X$ by $(X \setminus \{u\}) \cup \{v\}$ if $\delta((X \setminus \{u\}) \cup \{v\})<\delta(X)$ and  repeat this step  
until no further  improvement of the cardinality of the cut  is possible.
Then apply the above approach to determine a cut $\delta(X^{(k)})$ with  $|X^{(k)}|=k$ for any  $l\le k\le u$ and choose the best among 
the cuts $\delta(X^ {(k)})$, $l\le k\le u$. 
\subsubsection{A  heuristic for DAPT(G,d)}\label{subsub:heuristic}
Having described the heuristics for  MCBSSP let us turn back to  the $DAPT(G,d)$. 
The approach is presented in the form of a pseudo code in Algorithm~\ref{algorithm:ch} and  involves 
the heuristic solution of the MCBSSP as a subroutine (see pseudocode line 11). 
\smallskip

We first  consider  the question of determining the ``unused leaves'', i.e.\ leaves of the arrangement tree, into which no nodes of the guest graph are  arranged. 
Based on our  observations in the context of numerical tests we try to use as few  basic subtrees as possible to arrange all nodes of the guest graph.
Thus we collect the unused $b-n$ leaves (recall that $b:=|B|$ is the number of leaves of the host $d$-regular tree)     into as few basic subtrees  as possible. 
By considering that each basic subtree has  $b_1:=\frac{b}{d}$ we 
mark the  first $l_{uu} = \Big\lfloor\frac{b-n}{b_1}\Big\rfloor b_1$ leaves, or equivalently the first $\Big\lfloor\frac{b-n}{b_1}\Big\rfloor$ basic subtrees  
as {\em unused}\/ (see pseudocode lines~\ref{algorithm:costructionHeuristic:unusedLeavesBegin} -- \ref{algorithm:costructionHeuristic:unusedLeavesEnd}). 
Then we separate the vertices $X$ which will be placed on the leaves $b_{l_{uu} + 1}, \dots b_{l_{uu} + \frac{b}{d}}$, i.e.\ on the leaves of the first 
used basic subtree,  by solving  MCBSSP with the  parameters $l:=b_1-(b-n)\mod b_1$ and $u:=b_1$ (see pseudocode 
line~\ref{algorithm:costructionHeuristic:finAMinimumCardinalityCut}). This can be done by applying one of the heuristics described in 
Subsection~\ref{subsub:MCBSSP}. 
We repeat then this procedure $\Big\lceil \frac{n}{b_1} \Big\rceil - 1$ times to  obtain $\Big\lceil \frac{n}{b_1} \Big\rceil$ subproblems 
which are  solved recursively (pseudocode line~\ref{algorithm:costructionHeuristic:recursion}). 
The recursion calls will terminate when the height of the {\em arrangement tree}\/ becomes $1$; there an {\em arrangement} $\phi$ is selected at random.
\smallskip

\begin{algorithm}[htb]
	\caption{Construction heuristic.}
	\label{algorithm:ch}
	\begin{algorithmic}[1]
		\REQUIRE $G = (V, E)$ undirected graph and positive integer $d \in \mathbb{N}$ where $2 \leq d \leq n$; let be $|V| = n$ and $T$ the $d$-regular arrangement tree with the set of leaves $B$
		\ENSURE arrangement $\phi: V \to B$
		\STATE $h := \lceil \log_d{n} \rceil$ and $b := h^d$;
		\IF{$h = 1$}
			\STATE make the arrangement $\phi$ at random;
		\ELSE
			\STATE $l_{uu} := b - n$;
			\FOR{$i := 1$ \TO $d$}
				\IF{$l_{uu} \geq \frac{b}{d}$}
				\label{algorithm:costructionHeuristic:unusedLeavesBegin}
					\STATE $\phi^{-1}(l) := unused$, $(i - 1) \frac{b}{d} \leq l \leq i \frac{b}{d}$;
					\STATE $b_{uu} := b_{uu} - \frac{b}{d}$;
					\label{algorithm:costructionHeuristic:unusedLeavesEnd}
				\ELSE
					\STATE find a minimum cardinality cut $X \subset V(G)$ in graph $G$ subject to $\frac{b}{d} - b_{uu} \leq |X| \leq \frac{b}{d}$ by solving  MCBSSP with parameters $l:=\frac{b}{d} - l_{uu}$ and $u:=\frac{b}{d}$;
					\label{algorithm:costructionHeuristic:finAMinimumCardinalityCut}
					\STATE solve the problem for the graph $G[X]$ and a $d$-regular arrangement tree $T_X$ which height is $h - 1$ recursively; let $\phi_X$ be the solution of this recursive problem;
					\label{algorithm:costructionHeuristic:recursion}
					\STATE compute the inverse function of $\phi_X$ which we denote $\phi^{-1}_X$;
					\FOR{$j := 1$ \TO $\frac{b}{d}$}
						\STATE $\phi^{-1}((i - 1) \frac{b}{d} + j) := \phi^{-1}_X(j)$;
					\ENDFOR
					\STATE $G := G[G \backslash X]$;
					\STATE $l_{uu} := l_{uu} - (\frac{b}{d} - |X|)$;
				\ENDIF
			\ENDFOR
			\STATE compute the function $\phi$ from the function $\phi^{-1}$;
		\ENDIF
	\end{algorithmic}
\end{algorithm}
\smallskip

Now let us consider the  worst-case time complexity of the described approach. 
Let $f_{C}(n)$ denote the worst-case time complexity of the subroutine which solves  MCBSSP for a graph with $n$ vertices and any parameters $0 < l \leq u < n$. 
Since $n \leq b$ holds for all instances, the {\em worst-case time complexity}\/ of the whole algorithm is
\begin{equation}
	\label{equation:chTimeComplexity1}
	\begin{split}
		1			& \left(f_C\left(\frac{b}{d} d\right) + f_C\left(\frac{b}{d} (d - 1)\right) + \dots + f_C\left(\frac{b}{d} 2\right)\right) +\\
		d			& \left(f_C\left(\frac{\frac{b}{d}}{d} d\right) + f_C\left(\frac{\frac{b}{d}}{d} (d - 1)\right) + \dots + f_C\left(\frac{\frac{b}{d}}{d} 2\right)\right) +\\
					& \dots\\
		d^{h -2}	& \left(f_C\left(\frac{\frac{b}{d^{h - 2}}}{d} d\right) + f_C\left(\frac{\frac{b}{d^{h - 2}}}{d} (d - 1)\right) + \dots + f_C\left(\frac{\frac{b}{d^{h - 2}}}{d} 2\right)\right)\, ,\\
	\end{split}
\end{equation}
where the lines correspond to the {\em recursion depth}. Summarising we get the following worst case time complexity 
\begin{equation}
	\label{equation:chTimeComplexity2}
	\sum_{i = 0}^{h - 2}{d^i \sum_{j = 0}^{d - 2}{f_C\left(\frac{b}{d^{i + 1}} (d - j)\right)}}.
\end{equation}

For some particular heuristic to solve the  MCBSSP  we can substitute $f_C(n)$ by a  precise expression in (\ref{equation:chTimeComplexity2}).
Consider  the case of the   {\em local search}\/ based heuristic described in Subsection~\ref{subsub:MCBSSP}. 
When computing the cuts at   the first  recursion level  $u - l \leq \frac{b}{d}$ obviously holds.
If $X$ is the set of vertices generating the cut, then  $|X|\le \frac{b}{d}$ holds.  
When computing the $k$-th cut at the first recursion level  we have at most $\frac{b}{d} \frac{(d - k) b}{d}$ vertex pairs 
which could be exchanged and the cardinality of the cut after the pair-exchange can be computed in 
$O(\frac{b}{d}+(d-k)\frac{b}{d})=O(\frac{b}{d}(d-k+1))$ time. 
So  we get a  worst-case time complexity of 
$O\left(\frac{b}{d}\left(\frac{b}{d} \frac{(d - k) b}{d}\right) (\frac{b}{d}+\frac{(d-k)b}{d})\right)=O\left (\frac{b}{d}\left(\frac{b}{d} \frac{(d - k) b}{d}\right) \frac{b}{d}(d-k+1)\right )=O(\left (\frac{b}{d}\right )^4(d-k)(d-k+1)$ for the $k$-th cut in the first level 
(where the first factor in the above expression accounts for the number of $k$-$(n-k)$CP to be solved which is at most $u-l\le \frac{b}{d}$). 
Summarising for all cuts of the first level  we get 

\begin{equation}
	\label{equation:chTimeComplexity3}
O\left(\left(\frac{b}{d}\right)^4 \sum_{k=1}^{d-1} \left ((d - k)(d-k+1)\right ) 
\right )= O\left(\left(\frac{b}{d}\right)^4 \left (\sum_{i=1}^{d -1}i^2+\sum_{i=1}^{d -1} i\right )\right)=
O\left(\frac{b^4}{d}\right)\, .
\end{equation}

 Now let us  consider the recursion. After building the first $d - 1$ cuts we get $d$ subproblems each of them having most $\frac{b}{d}$ vertices. 
Thus  for the whole algorithm we get a time complexity $K$ with  
\begin{equation}
	\label{equation:chTimeComplexity7}
	K:=O\left(\frac{b^4}{d} + d \frac{\left(\frac{b}{d}\right)^4}{d} + d^2 \frac{\left(\frac{b}{d^2}\right)^4}{d} + \dots + d^{h - 2} \frac{\left(\frac{b}{d^{h - 2}}\right)^4}{d} + n\right).	
\end{equation}
Note that if the height of the {\em arrangement tree}\/ is $1$, the arrangement $\phi$ can be made at random and thus the {\em recursion depth}\/ is only $h - 2$. Using $d^h \frac{\left(\frac{b}{d^h}\right)^4}{d} = b \frac{\left(\frac{b}{b}\right)^4}{d} = \frac{b}{d}$, $d^{h - 1} \frac{\left(\frac{b}{d^{h - 1}}\right)^4}{d} = \frac{b}{d} \frac{\left(\frac{b}{\frac{b}{d}}\right)^4}{d} = b d^2$ and considering $b < n d$  we get
\begin{equation}
	\label{equation:chTimeComplexity8}
	K=O\left(\frac{b^4}{d}\sum_{i=0}^h \left (\frac{1}{d^3}\right )^i - b d^2 - \frac{b}{d} + n \right) = O(n^4 d).
\end{equation}

Now, we can state the following theorem.
\begin{theorem}
	\label{theorem:chTimeComplexity}
	The Algorithm~\ref{algorithm:ch} can be implemented with a worst case time complexity of  $O(n^4 d)$,   
if  the local search approach of Subsection~\ref{subsub:MCBSSP} is applied to  solve MCBSSP.
\end{theorem}

In fact the quality of this construction heuristic  depends significantly on the quality of the heuristic used to solve  MCBSSP. 
However, even if we were  able to solve MCBSSP to optimality, the construction  heuristic would  not necessarily compute  an optimal arrangement. 
As an example  consider $DAPT(G,2)$ with guest  graph $G$ as shown in   Figure~\ref{fig:chGraph}. 
Figure~\ref{fig:chNotAMinimumArrangement} shows an  arrangement obtained by the construction heuristic, 
where MCBSSP was always solved to optimality during the algorithm. 
This arrangement is not optimal; a strictly better arrangement is shown in Figure~\ref{fig:chAMinimumArrangement} (this is actually an optimal arrangement). 
The reason for this behaviour relies on the fact that  minimising the coefficients $a_i$, $i=1,2,\ldots,h$, 
 starting with $a_h$ and proceeding in the above order, does not necessarily lead to a minimum  value of $OV(G,d,\phi)$, see  (\ref{equation:ch}). 

In our computational experiment we observed that the construction heuristic which involved the pair-exchange approach to solve MCBSSP outperforms 
the heuristic which involves the approach of Feige et al.~\cite{Feige01}. 
Therefore in Section~\ref{numeric:sec} we just report on the performance of the more successful algorithm denoted by  CHLS, see also Section~\ref{sub:construct}.
\begin{figure}[htbp]
	\begin{center}
		\begin{tikzpicture}[
				scale=1,
				node/.style={circle, draw=black!100, fill=white!0, thick, inner sep=0pt, minimum size=5mm}
			]
		
			\node[node] (node1) at (0.00, 1.00) {1};
			\node[node] (node2) at (1.00, 2.00) {2};
			\node[node] (node3) at (1.00, 1.00) {3};
			\node[node] (node4) at (1.00, 0.00) {4};
			\node[node] (node5) at (2.00, 0.00) {5};
			\node[node] (node6) at (3.00, 1.00) {6};
			\node[node] (node7) at (3.00, 0.00) {7};
			
			\draw [-] (node1) to (node2);
			\draw [-] (node1) to (node3);
			\draw [-] (node1) to (node4);
			\draw [-] (node4) to (node5);
			\draw [-] (node5) to (node6);
			\draw [-] (node5) to (node7);
			\draw [-] (node6) to (node7);
		\end{tikzpicture}
	\end{center}
	\caption{\em Graph.}
	\label{fig:chGraph}
\end{figure}

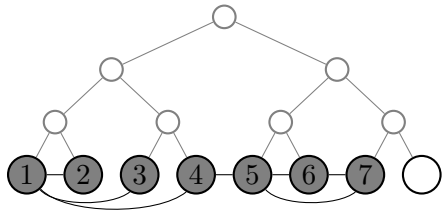
\begin{figure}[htbp]
	\begin{center}
		\begin{tikzpicture}[
				xscale=1.5, yscale=0.7,
				leafUsed/.style={circle, draw=black!100, fill=black!50, thick, inner sep=0pt, minimum size=5mm},
				leafUnused/.style={circle, draw=black!100, fill=white!100, thick, inner sep=0pt, minimum size=5mm},
				otherNode/.style={circle, draw=black!50, fill=white!100, thick, inner sep=0pt, minimum size=3mm}
			]
			\node[otherNode] (node0) at (1.75, 3.00) {};
	
			\node[otherNode] (node1) at (0.75, 2.00) {};
			\node[otherNode] (node2) at (2.75, 2.00) {};
	
			\node[otherNode] (node3) at (0.25, 1.00) {};
			\node[otherNode] (node4) at (1.25, 1.00) {};
			\node[otherNode] (node5) at (2.25, 1.00) {};
			\node[otherNode] (node6) at (3.25, 1.00) {};
	
			\node[leafUsed] (node7) at (0.00, 0.00) {1};
			\node[leafUsed] (node8) at (0.50, 0.00) {2};
			\node[leafUsed] (node9) at (1.00, 0.00) {3};
			\node[leafUsed] (node10) at (1.50, 0.00) {4};
			\node[leafUsed] (node11) at (2.00, 0.00) {5};
			\node[leafUsed] (node12) at (2.50, 0.00) {6};
			\node[leafUsed] (node13) at (3.00, 0.00) {7};
			\node[leafUnused] (node14) at (3.50, 0.00) {};

			\draw [-, black!50] (node0) to (node1);
			\draw [-, black!50] (node0) to (node2);
	
			\draw [-, black!50] (node1) to (node3);
			\draw [-, black!50] (node1) to (node4);
			\draw [-, black!50] (node2) to (node5);
			\draw [-, black!50] (node2) to (node6);
	
			\draw [-, black!50] (node3) to (node7);
			\draw [-, black!50] (node3) to (node8);
			\draw [-, black!50] (node4) to (node9);
			\draw [-, black!50] (node4) to (node10);
			\draw [-, black!50] (node5) to (node11);
			\draw [-, black!50] (node5) to (node12);
			\draw [-, black!50] (node6) to (node13);
			\draw [-, black!50] (node6) to (node14);
	
			\draw [-] (node7) to (node8);
			\draw [-] (node7) to [out=-60,in=-120] (node9);
			\draw [-] (node7) to [out=-60,in=-120] (node10);
			\draw [-] (node10) to (node11);
			\draw [-] (node11) to (node12);
			\draw [-] (node11) to [out=-60,in=-120] (node13);
			\draw [-] (node12) to (node13);
		\end{tikzpicture}
	\end{center}
	\caption{\em A non-optimal  arrangement $\phi$ with $OV(G,2,\phi)=26$ for $G$ in Figure~\ref{fig:chGraph}.}
	\label{fig:chNotAMinimumArrangement}
\end{figure}

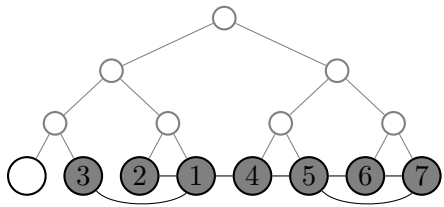
\begin{figure}[htbp]
	\begin{center}
		\begin{tikzpicture}[
				xscale=1.5, yscale=0.7,
				leafUsed/.style={circle, draw=black!100, fill=black!50, thick, inner sep=0pt, minimum size=5mm},
				leafUnused/.style={circle, draw=black!100, fill=white!100, thick, inner sep=0pt, minimum size=5mm},
				otherNode/.style={circle, draw=black!50, fill=white!100, thick, inner sep=0pt, minimum size=3mm}
			]
			\node[otherNode] (node0) at (1.75, 3.00) {};

			\node[otherNode] (node1) at (0.75, 2.00) {};
			\node[otherNode] (node2) at (2.75, 2.00) {};

			\node[otherNode] (node3) at (0.25, 1.00) {};
			\node[otherNode] (node4) at (1.25, 1.00) {};
			\node[otherNode] (node5) at (2.25, 1.00) {};
			\node[otherNode] (node6) at (3.25, 1.00) {};

			\node[leafUnused] (node7) at (0.00, 0.00) {};
			\node[leafUsed] (node8) at (0.50, 0.00) {3};
			\node[leafUsed] (node9) at (1.00, 0.00) {2};
			\node[leafUsed] (node10) at (1.50, 0.00) {1};
			\node[leafUsed] (node11) at (2.00, 0.00) {4};
			\node[leafUsed] (node12) at (2.50, 0.00) {5};
			\node[leafUsed] (node13) at (3.00, 0.00) {6};
			\node[leafUsed] (node14) at (3.50, 0.00) {7};

			\draw [-, black!50] (node0) to (node1);
			\draw [-, black!50] (node0) to (node2);

			\draw [-, black!50] (node1) to (node3);
			\draw [-, black!50] (node1) to (node4);
			\draw [-, black!50] (node2) to (node5);
			\draw [-, black!50] (node2) to (node6);

			\draw [-, black!50] (node3) to (node7);
			\draw [-, black!50] (node3) to (node8);
			\draw [-, black!50] (node4) to (node9);
			\draw [-, black!50] (node4) to (node10);
			\draw [-, black!50] (node5) to (node11);
			\draw [-, black!50] (node5) to (node12);
			\draw [-, black!50] (node6) to (node13);
			\draw [-, black!50] (node6) to (node14);

			\draw [-] (node8) to [out=-60,in=-120] (node10);
			\draw [-] (node9) to (node10);
			\draw [-] (node10) to (node11);
			\draw [-] (node11) to (node12);
			\draw [-] (node12) to (node13);
			\draw [-] (node12) to [out=-60,in=-120] (node14);
			\draw [-] (node13) to (node14);
		\end{tikzpicture}
	\end{center}
	\caption{\em An optimal arrangement $\phi$ with $OV(G,2,\phi)=24$ for $G$ in Figure~\ref{fig:chGraph}.}
	\label{fig:chAMinimumArrangement}
\end{figure}

\subsection{Local search approaches}
In this paragraph we propose two different local search heuristics for the DAPT. 
They can be used separately or also combined as described below. 
	 
\subsubsection{The pair-exchange heuristic}\label{subsub:PE}
 The algorithm starts with an arbitrary arrangement $\phi$ (it can be a random arrangement or an arrangement obtained by applying some other heuristic)
 and  tries to improve the objective function value by performing so-called pair-exchanges. 
More precisely the algorithm fixes an ordering of  the pairs of vertices $(v_i,v_j) \in V(G)\times V(G)$ with $v_i\neq v_j$ and checks whether  
a pair $(v_i,v_j)$ exists such 
that $OV(G,d,\phi^{\prime})< OV(G,d,\phi)$,  where $\phi^{\prime}$ is obtained from  $\phi$ by applying a pair-exchange:
	\begin{equation}\label{pairexch:equ}
\phi^{\prime}(v_k)=\left \{ \begin{array}{ll}
		\phi(v_j)		& \mbox{if } k=i\\
		\phi(v_i)		& \mbox{if } k=j\\
		\phi(v_k)		& \mbox{if } k \not\in \{i, j\}\end{array}\right. \mbox{ for } k\in \{1, \dots, n\} \, .
\end{equation}
If such a pair $(v_i,v_j)$ of vertices whose exchange  improves the objective function value can be found, 
then  $\phi$ is substituted by $\phi^ {\prime}$ and the procedure is iteratively repeated.
Otherwise the algorithm terminates and outputs the current arrangement. Note that this approach would keep unchanged the set of unused leaves. 
In order to be able to vary it  we work  with an extended  guest  graph $G' = (V', E')$  with  {\em vertex set} $V' = V \cup \{v_{n + 1}, \ldots v_{b}\}$ 
and  {\em edge set} $E' = E$, where $V$ is the vertex set of the original guest graph $G$. 
The new guest graph has as many vertices as the number of leaves of the host graph. Since the vertices $v_{n + 1}$, $\ldots$,  $v_{b}$ 
of  $V' \setminus  V$ are isolated vertices, $OV(G, d,\phi) = OV(G', d,\phi')$ obviously holds for all arrangements $\phi: V \to B$ and all 
arrangements $\phi': V' \to B$ such that $\phi(v) = \phi'(v)$ for all $v \in V$. 
Thus we can solve $DAPT(G',d)$ instead of solving $DAPT(G,d)$; an  optimal solution $\phi_{\ast}$ 
of $DAPT(G,d)$ is obtained from an optimal solution $\phi'_{\ast}$ of $DAPT(G',d)$ by setting $\phi_{\ast}(v)=\phi'_{\ast}(v)$, $\forall v \in V$. 
Note that, however, if the starting arrangement  is contiguous, then applying the pair-exchange to  $DAPT(G',d)$ instead of $DAPT(G,d)$   
can not generate  any variation in the set of used leaves. The reason is that 
\begin{equation}
	\label{equation:pe:1}
	d_T(b_i, b_j) \leq d_T(b_i, b_k) 
\end{equation}
holds for all $1 \leq i < j < k \leq b$ and thus a pair-exchange which  arranges  an  isolated vertex $v' \in V' \backslash V$ between some pair of  
eventually connected vertices can never improve the objective function value.

\begin{theorem}
	\label{theorem:peTimeComplexity}
	The pair-exchange heuristic for the $DAPT(G,d)$ can be implemented  with time complexity $O(n^2 d^2m\min\{m,n\} (\log{n}))$, 
 where $n$ is the number of vertices and $m$ is the number of edges in $G$. If the starting arrangement is contiguous, 
then the heuristic can be implemented in $O(n^2 m\min(m,n)(\log{n}))$ time.

		\begin{proof}
		There  are  $O(b^2)$ pairs of vertices in the graph $G'$. Since $2 m \leq OV(G, d, \phi) \leq 2 h m$ holds for every arrangement $\phi$, 
we can make at most $O(2 h m - 2 m) = O(h m) = O((\log{b}) m) = O(\log{(n d)} m) = O((\log{n} + \log{d}) m) = O(m \log{n})$ improvements of the objective function  value 
(if $d$ is considered to be a constant and by using  $b < n d$). 

Consider that the pairwise distances between all pairs of leaves in the arrangement tree can be computed  in $O(b^2\log{ n})=O(n^2d^2\log{n})$ time in a pre-processing step, see Observation~\ref{distance:obse}.
In order to update  the objective function value of an arrangement after a  pair-exchange of vertices $v_i$ and $v_j$ which transforms the current arrangement 
$\phi$ to the arrangement $\phi^{\prime}$ as in~(\ref{pairexch:equ}), 
 the length of the path between $\phi(v_i)$ ($\phi(v_j)$) and $\phi(v)$ is  substituted by  the length of the corresponding 
path between $\phi^{\prime}(v_i)$ ($\phi^{\prime}(v_j)$) and $\phi^{\prime}(v)$,  for all neighbours $v$ of $v_i$ ($v_j$).
Since the vertices which exchange position  have in total  $O(\min\{m,n\})$ neighbours, 
the objective function after a (candidate) pair-exchange can be updated in   $O(\min\{m,n\})$ time. With at most $O(b^2)$  (candidate) pair-exchanges to be performed in each  iteration 
  and at most  $O(m\log{n})$ iterations,
the overall time complexity of the algorithm amounts to $O(b^2 \min\{m,n\} m \log{n}) = O(n^2 d^2 m\min\{m,n\}\log{n})$.
	
		If the starting arrangement is contiguous, then just  $O(n^2)$ pairwise distances need to be computed in the pre-processing step and  the overall time complexity amounts to
$O(n^2 m\min\{m,n\}\log{n})$.
\qed
	\end{proof}
\end{theorem}

Clearly,  we can also fix an ordering of the pairs of  leaves and exchange the vertices arranged at some  pair of leaves (if any),  in this ordering.
 One would obtain a similar time complexity as in the general case of Theorem~\ref{theorem:peTimeComplexity}.  
We refer to these heuristics as {\sl vertex-based pair-exchange heuristic}\/ and {\sl leaf-based pair-exchange heuristic},  respectively.
Our computational experiments have shown that the vertex-based pair-exchange heuristic generally outperforms the leaf-based pair-exchange heuristic. For this reason we only report about the performance of the vertex-based pair-exchange heuristic (abbreviated by PEHVNA) in Section~\ref{numeric:sec}. 

\subsubsection{The shift-flip heuristic}\label{subsub:SL}
The last heuristic we discuss is the {\sl shift-flip heuristic}. First, we need two definitions.

\begin{definition}[{\sl Flip}]
	\label{definition:flipToFlip}
	Let $G = (V, E)$ be an undirected guest graph with $|V| = n$, $T$ a $d$-regular  tree,  with $2 \leq d \leq n$,  and let $B$ be  the  set of leaves of $T$. 
 Let $ \phi: V \to B$ be an arrangement. 
Further, let $e,g,l,r \in \mathbb{N}\cup\{0\}$,  be parameters with  $0 \leq e < h$,  $1\le g \le d^e$,   $1 \leq l < r \leq d$.
Finally let  $f$ be a bijection  $f: B \to B$ defined as follows:
	\begin{equation}
		\label{equation:definitionFlipToFlip}
		f(b_i) =\left \{ 
			\begin{array}{ll}
				b_{\Delta(g)+(r - 1) d^{h - (e + 1)}+t_i}	& \text{for } i= \Delta(g)+(l - 1) d^{h - (e + 1)} + t_i\, ,  1\le  t_i \leq d^{h - (e + 1)}\\
				b_{\Delta(g)+(l - 1) d^{h - (e + 1)}+t_i}	& \text{for }  i= \Delta(g)+(r - 1) d^{h - (e + 1)} + t_i\, ,  1\le  t_i \leq d^{h - (e + 1)} \\
				b_i					& \text{otherwise}\\
			\end{array} \right .
			\hspace*{-0.2cm} , 
	\end{equation}
where $\Delta(g):=(g-1)d^{h-e}$. 
	The arrangement $\phi_f: V \to B$ where $\phi_f = f \circ \phi$ is a {\sl flip of the arrangement $\phi$}. We say that we {\sl flip}\/ the arrangement $\phi$
at the $l$-th and $r$-th $d$-regular subtrees of the  $g$-th node in  level $e$.  
\end{definition}

In a more descriptive explanation  a flip consists of interchanging the preimages of the leaves of two $d$-regular subtrees of the arrangement tree which have the same
 height and whose roots have a common  father vertex, while preserving  the order of the leaves in each of the two interchanged subtrees. 
More precisely   we consider the vertices of the $d$-regular tree as being partitioned into levels, the root having level $0$, its $d$ sons having level $1$ and so on, to end up with the leaves at level $h-1$. In Definition~\ref{definition:flipToFlip} we consider the $g$-th  vertex  in level $e$ and the indices   $l$ and $r$ of two sons of that vertex. The successors of each of those suns  build a $d$-regular subtree of height $h-(e+1)$, respectively. The flip operation interchanges exactly 
the preimages of the leaves of these two $d$-regular subtrees by preserving in each subtree the order of the leaves induced by the canonical order of the leaves in $T$. 
\smallskip

For  an illustration consider an instance $DAPT(G,d)$ with guest graph $G$ given in Figure~\ref{fig:flipGraph} and  $d = 3$. 
\medskip

\begin{figure}[htbp]
	\begin{center}
		\begin{tikzpicture}[
				scale=1,
				node/.style={circle, draw=black!100, fill=white!0, thick, inner sep=0pt, minimum size=5mm}
			]
		
			\node[node] (node1) at (0.00, 0.00) {1};
			\node[node] (node2) at (1.00, 0.00) {2};
			\node[node] (node3) at (2.00, 0.00) {3};
			\node[node] (node4) at (3.00, 0.00) {4};
			\node[node] (node5) at (4.00, 0.00) {5};
			\node[node] (node6) at (5.00, 0.00) {6};
			\node[node] (node7) at (6.00, 0.00) {7};
			\node[node] (node8) at (7.00, 0.00) {8};
			\node[node] (node9) at (8.00, 0.00) {9};
			\node[node] (node10) at (9.00, 0.00) {10};
	
			\draw [-] (node1) to (node2);
			\draw [-] (node2) to (node3);
			\draw [-] (node3) to (node4);
			\draw [-] (node4) to (node5);
			\draw [-] (node5) to (node6);
			\draw [-] (node6) to (node7);
			\draw [-] (node7) to (node8);
			\draw [-] (node8) to (node9);
			\draw [-] (node9) to (node10);
		\end{tikzpicture}
	\end{center}
	\caption{\em A guest graph.}
	\label{fig:flipGraph}
\end{figure}
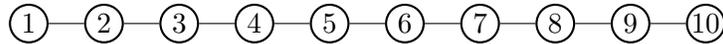
\smallskip

Consider further  an arrangement represented  in Figure~\ref{fig:flipArrangement}; each filled leaf contains the index of the vertex of $G$ 
mapped into that leaf.
\smallskip

\begin{figure}[htbp]
	\begin{center}
		\begin{tikzpicture}[
				xscale=1, yscale=0.7,
				leafUsed/.style={circle, draw=black!100, fill=black!50, thick, inner sep=0pt, minimum size=5mm},
				leafUnused/.style={circle, draw=black!100, fill=white!100, thick, inner sep=0pt, minimum size=5mm},
				otherNode/.style={circle, draw=black!50, fill=white!100, thick, inner sep=0pt, minimum size=3mm}
			]
			\node[otherNode] (node0) at (7.02, 3.00) {};

			\node[otherNode] (node1) at (2.16, 2.00) {};
			\node[otherNode] (node2) at (7.02, 2.00) {};
			\node[otherNode] (node3) at (11.88, 2.00) {};
			
			\node[otherNode] (node4) at (0.54, 1.00) {};
			\node[otherNode] (node5) at (2.16, 1.00) {};
			\node[otherNode] (node6) at (3.78, 1.00) {};
			\node[otherNode] (node7) at (5.40, 1.00) {};
			\node[otherNode] (node8) at (7.02, 1.00) {};
			\node[otherNode] (node9) at (8.64, 1.00) {};
			\node[otherNode] (node10) at (10.26, 1.00) {};
			\node[otherNode] (node11) at (11.88, 1.00) {};
			\node[otherNode] (node12) at (13.50, 1.00) {};

			\node[leafUsed] (node13) at (0.00, 0.00) {1};
			\node[leafUsed] (node14) at (0.54, 0.00) {2};
			\node[leafUnused] (node15) at (1.08, 0.00) {};
			\node[leafUnused] (node16) at (1.62, 0.00) {};
			\node[leafUnused] (node17) at (2.16, 0.00) {};
			\node[leafUnused] (node18) at (2.70, 0.00) {};
			\node[leafUnused] (node19) at (3.24, 0.00) {};
			\node[leafUnused] (node20) at (3.78, 0.00) {};
			\node[leafUsed] (node21) at (4.32, 0.00) {3};
			\node[leafUnused] (node22) at (4.86, 0.00) {};
			\node[leafUnused] (node23) at (5.40, 0.00) {};
			\node[leafUsed] (node24) at (5.94, 0.00) {4};
			\node[leafUsed] (node25) at (6.48, 0.00) {5};
			\node[leafUnused] (node26) at (7.02, 0.00) {};
			\node[leafUsed] (node27) at (7.56, 0.00) {6};
			\node[leafUsed] (node28) at (8.10, 0.00) {7};
			\node[leafUsed] (node29) at (8.64, 0.00) {8};
			\node[leafUnused] (node30) at (9.18, 0.00) {};
			\node[leafUsed] (node31) at (9.72, 0.00) {9};
			\node[leafUsed] (node32) at (10.26, 0.00) {10};
			\node[leafUnused] (node33) at (10.80, 0.00) {};
			\node[leafUnused] (node34) at (11.34, 0.00) {};
			\node[leafUnused] (node35) at (11.88, 0.00) {};
			\node[leafUnused] (node36) at (12.42, 0.00) {};
			\node[leafUnused] (node37) at (12.96, 0.00) {};
			\node[leafUnused] (node38) at (13.50, 0.00) {};
			\node[leafUnused] (node39) at (14.04, 0.00) {};

			\draw [-, black!50] (node0) to (node1);
			\draw [-, black!50] (node0) to (node2);
			\draw [-, black!50] (node0) to (node3);

			\draw [-, black!50] (node1) to (node4);
			\draw [-, black!50] (node1) to (node5);
			\draw [-, black!50] (node1) to (node6);
			\draw [-, black!50] (node2) to (node7);
			\draw [-, black!50] (node2) to (node8);
			\draw [-, black!50] (node2) to (node9);
			\draw [-, black!50] (node3) to (node10);
			\draw [-, black!50] (node3) to (node11);
			\draw [-, black!50] (node3) to (node12);

			\draw [-, black!50] (node4) to (node13);
			\draw [-, black!50] (node4) to (node14);
			\draw [-, black!50] (node4) to (node15);
			\draw [-, black!50] (node5) to (node16);
			\draw [-, black!50] (node5) to (node17);
			\draw [-, black!50] (node5) to (node18);
			\draw [-, black!50] (node6) to (node19);
			\draw [-, black!50] (node6) to (node20);
			\draw [-, black!50] (node6) to (node21);
			\draw [-, black!50] (node7) to (node22);
			\draw [-, black!50] (node7) to (node23);
			\draw [-, black!50] (node7) to (node24);
			\draw [-, black!50] (node8) to (node25);
			\draw [-, black!50] (node8) to (node26);
			\draw [-, black!50] (node8) to (node27);
			\draw [-, black!50] (node9) to (node28);
			\draw [-, black!50] (node9) to (node29);
			\draw [-, black!50] (node9) to (node30);
			\draw [-, black!50] (node10) to (node31);
			\draw [-, black!50] (node10) to (node32);
			\draw [-, black!50] (node10) to (node33);
			\draw [-, black!50] (node11) to (node34);
			\draw [-, black!50] (node11) to (node35);
			\draw [-, black!50] (node11) to (node36);
			\draw [-, black!50] (node12) to (node37);
			\draw [-, black!50] (node12) to (node38);
			\draw [-, black!50] (node12) to (node39);

			\draw [-] (node13) to [out=-60,in=-120] (node14);
			\draw [-] (node14) to [out=-40,in=-140] (node21);
			\draw [-] (node21) to [out=-60,in=-120] (node24);
			\draw [-] (node24) to [out=-60,in=-120] (node25);
			\draw [-] (node25) to [out=-60,in=-120] (node27);
			\draw [-] (node27) to [out=-60,in=-120] (node28);
			\draw [-] (node28) to [out=-60,in=-120] (node29);
			\draw [-] (node29) to [out=-60,in=-120] (node31);
			\draw [-] (node31) to [out=-60,in=-120] (node32);
			
			\draw [dashed] (13.50, 2.00) -- (0.54, 2.00) node[above] {$e = 1$};
			\draw [<-] (node2) -- (6.42, 2.25) node[above] {$g = 2$};
			\draw [<-] (node8) -- (6.42, 1.25) node[above] {$l = 2$};
			\draw [<-] (node9) -- (8.04, 1.25) node[above] {$r = 3$};
			
			\draw [decorate, decoration={brace, mirror}] (6.21, -0.60) -- (7.83, -0.60);
			\draw [decorate, decoration={brace, mirror}] (7.83, -0.60) -- (9.45, -0.60);
			\draw [<->] (7.02, -0.80) to [out=-60,in=-120] (8.64, -0.80);
		\end{tikzpicture}
	\end{center}
	\caption{\em An arrangement $\phi$ with $OV(G, 3, \phi) = 32$ for $G$ in Figure~\ref{fig:flipGraph}.}
	\label{fig:flipArrangement}
\end{figure}
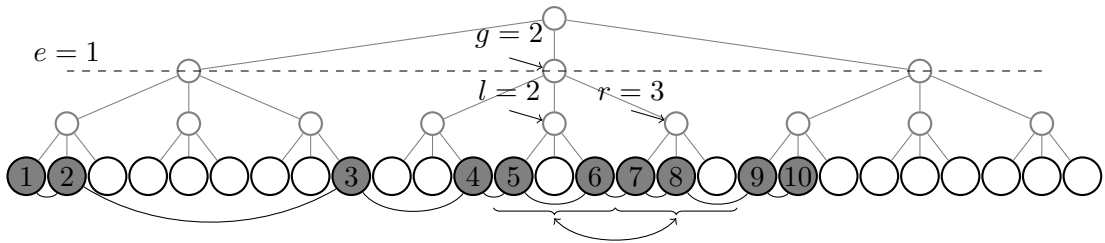

  In Figure~\ref{fig:flipArrangementFlipped} we see the flip obtained from the 
 arrangement represented in Figure~\ref{fig:flipArrangement} with parameters  $e = 1$, $g = 2$, $l = 2$ and $r = 3$. 
Note that flipping does not change the  objective  function value of the arrangement.

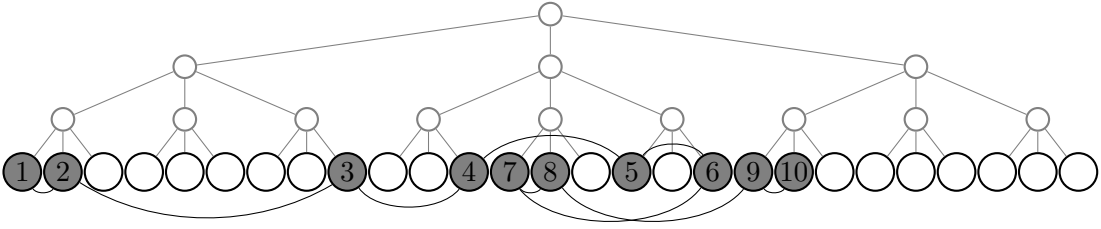
\begin{figure}[htbp]
	\begin{center}
		\begin{tikzpicture}[
				xscale=1, yscale=0.7,
				leafUsed/.style={circle, draw=black!100, fill=black!50, thick, inner sep=0pt, minimum size=5mm},
				leafUnused/.style={circle, draw=black!100, fill=white!100, thick, inner sep=0pt, minimum size=5mm},
				otherNode/.style={circle, draw=black!50, fill=white!100, thick, inner sep=0pt, minimum size=3mm}
			]
			\node[otherNode] (node0) at (7.02, 3.00) {};

			\node[otherNode] (node1) at (2.16, 2.00) {};
			\node[otherNode] (node2) at (7.02, 2.00) {};
			\node[otherNode] (node3) at (11.88, 2.00) {};
			
			\node[otherNode] (node4) at (0.54, 1.00) {};
			\node[otherNode] (node5) at (2.16, 1.00) {};
			\node[otherNode] (node6) at (3.78, 1.00) {};
			\node[otherNode] (node7) at (5.40, 1.00) {};
			\node[otherNode] (node8) at (7.02, 1.00) {};
			\node[otherNode] (node9) at (8.64, 1.00) {};
			\node[otherNode] (node10) at (10.26, 1.00) {};
			\node[otherNode] (node11) at (11.88, 1.00) {};
			\node[otherNode] (node12) at (13.50, 1.00) {};

			\node[leafUsed] (node13) at (0.00, 0.00) {1};
			\node[leafUsed] (node14) at (0.54, 0.00) {2};
			\node[leafUnused] (node15) at (1.08, 0.00) {};
			\node[leafUnused] (node16) at (1.62, 0.00) {};
			\node[leafUnused] (node17) at (2.16, 0.00) {};
			\node[leafUnused] (node18) at (2.70, 0.00) {};
			\node[leafUnused] (node19) at (3.24, 0.00) {};
			\node[leafUnused] (node20) at (3.78, 0.00) {};
			\node[leafUsed] (node21) at (4.32, 0.00) {3};
			\node[leafUnused] (node22) at (4.86, 0.00) {};
			\node[leafUnused] (node23) at (5.40, 0.00) {};
			\node[leafUsed] (node24) at (5.94, 0.00) {4};
			\node[leafUsed] (node25) at (6.48, 0.00) {7};
			\node[leafUsed] (node26) at (7.02, 0.00) {8};
			\node[leafUnused] (node27) at (7.56, 0.00) {};
			\node[leafUsed] (node28) at (8.10, 0.00) {5};
			\node[leafUnused] (node29) at (8.64, 0.00) {};
			\node[leafUsed] (node30) at (9.18, 0.00) {6};
			\node[leafUsed] (node31) at (9.72, 0.00) {9};
			\node[leafUsed] (node32) at (10.26, 0.00) {10};
			\node[leafUnused] (node33) at (10.80, 0.00) {};
			\node[leafUnused] (node34) at (11.34, 0.00) {};
			\node[leafUnused] (node35) at (11.88, 0.00) {};
			\node[leafUnused] (node36) at (12.42, 0.00) {};
			\node[leafUnused] (node37) at (12.96, 0.00) {};
			\node[leafUnused] (node38) at (13.50, 0.00) {};
			\node[leafUnused] (node39) at (14.04, 0.00) {};

			\draw [-, black!50] (node0) to (node1);
			\draw [-, black!50] (node0) to (node2);
			\draw [-, black!50] (node0) to (node3);

			\draw [-, black!50] (node1) to (node4);
			\draw [-, black!50] (node1) to (node5);
			\draw [-, black!50] (node1) to (node6);
			\draw [-, black!50] (node2) to (node7);
			\draw [-, black!50] (node2) to (node8);
			\draw [-, black!50] (node2) to (node9);
			\draw [-, black!50] (node3) to (node10);
			\draw [-, black!50] (node3) to (node11);
			\draw [-, black!50] (node3) to (node12);

			\draw [-, black!50] (node4) to (node13);
			\draw [-, black!50] (node4) to (node14);
			\draw [-, black!50] (node4) to (node15);
			\draw [-, black!50] (node5) to (node16);
			\draw [-, black!50] (node5) to (node17);
			\draw [-, black!50] (node5) to (node18);
			\draw [-, black!50] (node6) to (node19);
			\draw [-, black!50] (node6) to (node20);
			\draw [-, black!50] (node6) to (node21);
			\draw [-, black!50] (node7) to (node22);
			\draw [-, black!50] (node7) to (node23);
			\draw [-, black!50] (node7) to (node24);
			\draw [-, black!50] (node8) to (node25);
			\draw [-, black!50] (node8) to (node26);
			\draw [-, black!50] (node8) to (node27);
			\draw [-, black!50] (node9) to (node28);
			\draw [-, black!50] (node9) to (node29);
			\draw [-, black!50] (node9) to (node30);
			\draw [-, black!50] (node10) to (node31);
			\draw [-, black!50] (node10) to (node32);
			\draw [-, black!50] (node10) to (node33);
			\draw [-, black!50] (node11) to (node34);
			\draw [-, black!50] (node11) to (node35);
			\draw [-, black!50] (node11) to (node36);
			\draw [-, black!50] (node12) to (node37);
			\draw [-, black!50] (node12) to (node38);
			\draw [-, black!50] (node12) to (node39);

			\draw [-] (node13) to [out=-60,in=-120] (node14);
			\draw [-] (node14) to [out=-40,in=-140] (node21);
			\draw [-] (node21) to [out=-60,in=-120] (node24);
			\draw [-] (node24) to [out=50,in=130] (node28);
			\draw [-] (node28) to [out=60,in=120] (node30);
			\draw [-] (node30) to [out=-120,in=-60] (node25);
			\draw [-] (node25) to [out=-60,in=-120] (node26);
			\draw [-] (node26) to [out=-60,in=-120] (node31);
			\draw [-] (node31) to [out=-60,in=-120] (node32);
		\end{tikzpicture}
	\end{center}
	\caption{\em A flip of the arrangement shown in Figure~\ref{fig:flipArrangement}, for the guest graph shown in Figure~\ref{fig:flipGraph} and $d = 3$. The parameters of the flip are $e=1$, $g=2$, $l=2$ and $r=3$. The objective value of the flipped arrangement remains unchanged and equals $32$.}
	\label{fig:flipArrangementFlipped}
\end{figure}

\begin{proposition}
Let $G = (V, E)$ be an undirected guest graph with $|V| = n$,
 $T$ a $d$-regular  tree, with $2 \leq d \leq n$ and  let  $B$ be  the set of  leaves of $T$.   
Further, let $e,g,l,r \in \mathbb{N}\cup\{0\}$,  be parameters with  $0 \leq e < h$,  $1\le g \le d^e$,   $1 \leq l < r \leq d$. 
Let $f$ be a bijective function $f: B \to B$ defined as in Definition~\ref{definition:flipToFlip}. 
For any arrangement  $ \phi: V \to B$ and the corresponding    flip $\phi_f: V \to B$ of the arrangement  $\phi$,    
$\phi_f = f \circ \phi$,  
the equality  $OV(G,d,\phi_f)=OV(G,d,\phi)$ holds.
\end{proposition}
\begin{proof}
To prove the statement we  make use of Observation~\ref{distance:obse}. For $v_i \in V$, $i=1,2,\ldots, n$,   let us denote 
by $p(i)$, $p_f(i)$ the indices of the leaves $\phi(v_i)$, $\phi_f(v_i)$  of $T$ in the canonical ordering, respectively. 
We clearly have $p(i),p_f(i) \in 1,2,\ldots, d^h$, for all $i=1,2,\ldots,n$. 
According to Observation~\ref{distance:obse} we get the following expressions for the objective function values of $\phi$ and $\phi_f$: 

\begin{eqnarray}
\label{objfnc:equ}
OV(G,d,\phi)&= &\sum_{\{v_i,v_j\}\in E(G)}d_T\left (\phi(v_i),\phi(v_j)\right )\\
\nonumber  
 &=& \sum_{\{v_i,v_j\}\in E(G)} 2\mbox{argmin}\left \{k \in \{1,2,\ldots,h\}\colon \left \lfloor \frac{p(i)-1}{d^k} \right \rfloor=\left \lfloor \frac{p(j)-1}{d^k} \right \rfloor\right \}
\end{eqnarray}
and 
\begin{eqnarray}
\label{objfnc_f:equ} 
OV(G,d,\phi_f)&=&\sum_{\{v_i,v_j\}\in E(G)}d_T\left (\phi_f(v_i),\phi_f(v_j)\right )\\
\nonumber
&=& \sum_{\{v_i,v_j\}\in E(G)} 2\mbox{argmin}\left \{k \in \{1,2,\ldots,h\}\colon \left \lfloor \frac{p_f(i)-1}{d^k} \right \rfloor= \left \lfloor \frac{p_f(j)-1}{d^k} \right \rfloor\right \}
\end{eqnarray}

 Consider  the index $p$ of an arbitrary leaf $b_p$ of $T$ (in the canonical order) 
 written as $p= (u-1)d^{h-e}+ (s - 1) d^{h - (e + 1)} + t$ for some natural numbers 
 $1\le u\le d^{e}$, $1\le s\le d$ and $1\le  t\le d^ {h-(e+1)}$.
$u$  represents the index of the unique node $x$ at level $e$ which is an ancestor of $b_p$, $s$ represents the index of the $d$-regular subtree $T_1$ of height $h-(e+1)$
hanging  on $x$ and $t$ represents the index of $b_p$ in $T_1$ according to the canonical order of the leaves of  $T_1$ induced  by the canonical order of the leaves of $T$. Then the following equality holds 

\begin{equation} \label{division}
\left \lfloor \frac{p-1}{d^k}\right \rfloor=\left \{ \begin{array}{ll}
(u-1)d^{h-e-k} + (s-1)d^{h-(e+1)-k}+\left \lfloor\frac{t-1}{d^k}\right \rfloor  &\mbox{ if } k<h-(e+1)\\[0.2cm]
(u-1) &  \mbox{ if } k=h-e\\[0.2cm]
\left \lfloor \frac{u}{d^{k-(h-e)}} \right \rfloor & \mbox{ if } k>h-e\\[0.3cm]
(u-1)d+(s-1) & \mbox{ if } k=h-(e+1)\end{array}\right .\, ,\end{equation}
 for any $1\le u\le d^e$, any $1\le s\le d$ and any $1\le t\le d^{h-(e+1)}$.
 Notice that according to  Definition~\ref{definition:flipToFlip} $\phi(v_i)\neq \phi_f(v_i)$ holds,  only if $p(i)= \Delta(g)+(l - 1) d^{h - (e + 1)} + t_i$ or  
 $p(i)= \Delta(g)+(r - 1) d^{h - (e + 1)} + t_i$ with some $ 1\le  t_i \leq d^{h - (e + 1)}$. 
Moreover, the following two implications hold  for $t_i =1,2,\ldots, d^{h - (e + 1)}$:
\begin{equation}\label{implication1}
   \mbox{$p(i)= \Delta(g)+(l - 1) d^{h - (e + 1)} + t_i$ implies  $p_f(i)=\Delta(g)+(r - 1) d^{h - (e + 1)} + t_i$,}
\end{equation}

\begin{equation} \label{implication2} 
\mbox{$p(i)= \Delta(g)+(r - 1) d^{h - (e + 1)} + t_i$  implies  $p_f(i)=\Delta(g)+(l - 1) d^{h - (e + 1)} + t_i$.}\end{equation}
\smallskip

 Consider now an edge $(v_i,v_j)$ with $\phi(v_i)\neq \phi_f(v_i)$ or $\phi(v_j)\neq \phi_f(v_j)$, which is equivalent to $p(i)\neq p_f(i)$ or $p(j)\neq p_f(j)$. 
There are two cases:  (I) $p(i)\neq p_f(i)$ and  $p(j)\neq p_f(j)$, or (II) just one of the inequalities $p(i)\neq p_f(i)$, $p(j)\neq p_f(j)$ holds. 
\smallskip

{\bf Case I.} In this case one of the following cases can happen:
\begin{description}
\item{Case Ia.}
 $p(i)= \Delta(g)+(l - 1) d^{h - (e + 1)} + t_i$ and   $p(j)= \Delta(g)+(l - 1) d^{h - (e + 1)} + t_j$, or
\item{Case Ib.}  $p(i)= \Delta(g)+(r - 1) d^{h - (e + 1)} + t_i$ and
 $p(j)= \Delta(g)+(r - 1) d^{h - (e + 1)} + t_j$, or 
 \item{Case Ic.}
 $p(i)= \Delta(g)+(l - 1) d^{h - (e + 1)} + t_i$ and   $p(j)= \Delta(g)+(r - 1) d^{h - (e + 1)} + t_j$, or
\item{Case Id.}  $p(i)= \Delta(g)+(r - 1) d^{h - (e + 1)} + t_i$ and
 $p(j)= \Delta(g)+(l - 1) d^{h - (e + 1)} + t_j$.
\end{description} 

In Case Ic and in Case Id we get $d(\phi(i),\phi(j))=d(\phi_f(i),\phi_f(j))=2(h-e)$ by  applying (\ref{division}) 
and considering (\ref{implication1}), (\ref{implication2}). In Case Ia and in Case Ib we get 
\[ 
d(\phi(i),\phi(j))=d(\phi_f(i),\phi_f(j))=
\]
\[2 \min\left \{ h-(e+1), \mbox{argmin} \Big \{k\in \{1,2,h-(e+2)\}\colon \frac{t_i-1}{d^k}=\frac{t_j-1}{d^k} \Big\} \right \}\, .\]
 \smallskip

{\bf Case II.} Assume w.l.o.g.\ that $p(i)=(g-1)d^{h-e}+ (l - 1) d^{h - (e + 1)} + t_i$ 
and let $p(j)=(u-1)d^{h-e}+ (s - 1) d^{h - (e + 1)} + t_j$, where $g\neq u$ or  $s\not\in\{l,r\}$.
Clearly $p_f(i)=(g-1)d^{h-e}+ (r - 1) d^{h - (e + 1)} + t_i$ and $p_f(j)=p(j)=(u-1)d^{h-e}+ (s - 1) d^{h - (e + 1)} + t_j$. 
If $u=g$ and $s\not\in \{l,r\}$,  then (\ref{division}) together with  Observation~\ref{distance:obse} implies $d_T(\phi(i)\phi(j))=d_T(\phi_f(i)\phi_f(j))=2(h-e)$.

\noindent  Otherwise, if $u\neq g$,  then (\ref{division}) implies 
$\left\lfloor \frac{p(i)-1}{d^k}\right\rfloor = \left\lfloor \frac{p_f(i)-1}{d^k}\right\rfloor$ for all $k\ge h-e$ and $$\left\lfloor \frac{p(i)-1}{d^k}\right\rfloor \neq \left\lfloor \frac{p(j)-1}{d^k}\right\rfloor \, \text{;} \, \left\lfloor \frac{p_f(i)-1}{d^k}\right\rfloor\neq  \left\lfloor \frac{p(j)-1}{d^k}\right\rfloor=\left\lfloor \frac{p_f(j)-1}{d^k}\right\rfloor\, ,  \mbox{for all $k < h - e$,}$$ which together with Observation~\ref{distance:obse} implies then $ d_T(\phi(i),\phi(j))=d_T(\phi_f(i),\phi_f(j))$. 

\noindent Thus             $ d_T(\phi(i)\phi(j))=d_T(\phi_f(i)\phi_f(j))$ for any edge $(v_i,v_j)\in E$. Therefore the right-hand sides of the equations   (\ref{objfnc:equ}) and (\ref{objfnc_f:equ}) are equal and $OV(G,d,\phi)=OG(G,d,\phi_f)$.\qed
\end{proof}

\begin{definition}[{\sl shift}]
	\label{definition:shiftToShift}
	Let $G = (V, E)$ be an undirected guest graph with  $|V| = n$ and  $T$ a $d$-regular arrangement tree with  $2 \leq d \leq n$, set of leaves $B$ and number of leaves $b=|B|$. 
Let  $ \phi: V \to B$ be  an arrangement. Further, let $k \in \mathbb{N}$ be an integer. An arrangement $\phi_k$ with
	\begin{equation}				
		\label{equation:shiftToShift}
		\phi_k(v) := b_{(((i - 1) + k) \mod b) + 1}\, ,  \text{ where } \phi(v) = b_i\, ,
	\end{equation}
	is a {\sl shift}\/ of the arrangement $\phi$. We say that we {\sl shift}\/ the arrangement $\phi$ by $k$.
\end{definition}

The idea of the shift-flip  heuristic is fairly simple. For a given arrangement $\phi$  we find out a $1 \leq k \leq b$ which minimises the  objective 
function value $OV(G, d, \phi_k)$. 
There are  two possibilities to define the  shift step. In the first variant we apply the  shift by $k$ defined as above only if it implies  an improvement 
of the objective function value, i.e.\ $OV(G, d, \phi_k)< OV(G, d, \phi)$,  and substitute then the current arrangement $\phi$ by the improved one $\phi_k$. 
In the second variant we also accept an arrangement which keeps the objective function value unchanged. If no such an arrangement can be found, then a further 
flip is performed. Both approaches proceed  in the next iteration  by applying a random flip to the current  arrangement and so on  until  a  termination 
criterion is satisfied. Both variants of the  heuristic output the best arrangement found during the search.  
We report about the performance of the second variant because this variant seems to outperform the first one.

 Of course there are a number  of  possibilities to define a  terminating criterion. It can be a run time bound which defines the maximum length of  a time interval  the algorithm is allowed to  run without doing an improvement. Or it can be a bound on the overall number of  flip and shift steps  performed without  improving the   objective function value.

Both variants of the shift-flip heuristic (SF)  can be combined with the  pair-exchange heuristic (PE). Since the search neighbourhoods of the two  heuristics   are significantly different, it is possible to escape from the local minima of SF by just applying a search in the PE neighbourhood and vice-versa.

\section{Test instances}
\label{TestInst:sec}
We test and  compare the above described heuristics on some families of test instances. To the best of our knowledge there are no standard test instances  for this problem, so we have generated some  test instances ourselves.   We introduce the following  families  of test instances which are also available at \url{http://www.opt.math.tu-graz.ac.at/~cela/public.htm}.

\subsection{Test instances solvable  by complete enumeration}
The guest  graphs of these instances are marked by the prefix {\em ``CE\_''}. The first graph in this category {\em CE\_sample}\/ corresponds 
to the graph in Figure~\ref{fig:graph}. Further  we consider  2 (thin) graphs {\em CE\_thin7}\/ ($n = 7$, $m = 7$) and 
{\em CE\_thin10}\/ ($n = 10$, $m = 11$) with $7$ and $10$ vertices,  respectively.  
We generate test instances with guest graph  {\em CE\_thin7}\/ and all possible values of $d$,  $2\le d\le 7$. With the guest graph  
{\em CE\_thin10}\/ we generate instances with  $d = 2$ and $d = 4$. Further we consider   some analogous instances with denser guest  graphs:  
{\em CE\_dense7}\/ ($n = 7$, $m = 14$) and {\em CE\_dense10}\/ ($n = 10$, $m = 26$) with $7$ and $10$ vertices.
Finally, we consider  a $3 \times 3$ mesh {\em CE\_mesh9}\/ and  $d = 2$, $d = 3$ and $d = 4$.

For this family of test instances the  precise values of $d$ were chosen so as to be able to solve these instances by complete enumeration within a prespecified time limit, see Section~\ref{numeric:sec}.    

\subsection{Test instances with known optimal solution}
These instances are special cases of the  DAPT which can be solved by a polynomial time algorithm, see \cite{CeSta13,Sta12}. 
The guest  graphs of these instances are marked  by the prefix {\em ``SC\_''}.  Unless the special case involves a particular choice of $d$, 
we use $d = 2$ and $d = 7$ for all considered guest graphs. We consider  instances of following types:
\begin{itemize}
	\item Instances for  which  $d = n - 1$ where $n$ is the number of vertices of the guest graph. 
We use  $3$ guest graphs generated at random for this class of instances: {\em SC\_random25}, {\em SC\_random50}\/ and {\em SC\_random75}. 
These graphs  have the same number of vertices,  $n=500$,  and in each of them  any pair of non-equal vertices build an edge independently 
at random  with probability  $0.25$, $0.50$ and $0.75$, respectively.
	\item Instances whose guest graphs build  a {\em star}, that is they consist just  of a {\sl central vertex}\/ connected by 
an edge to all other vertices of the graph. The concrete graphs are $SC\_star50$, $SC\_star500$ and $SC\_star1000$ with $50$, $500$ and 
$1000$ vertices,  respectively.
	\item The guest graph in Figure~\ref{fig:noMinimumContinuousArrangementExistsGraph} and the choice $d = 4$. This guest graph is 
an extended star and this instance is referred to as $SC\_extStar$. 
	\item Instances whose guest graphs build  a $d${\em -regular tree}. We denote these guest graphs/instances  by {\em SC\_treeDGxHy}\/ 
where $x = d$ holds and $y$ is the height of the tree.
	\item Instances whose guest graphs build of a {\em path}. We denote these guest  graphs  by  {\em SC\_path50}, {\em SC\_path500}\/ 
and {\em SC\_path1000}. They have $50$, $500$ and $1000$ vertices,  respectively.
	\item Instances whose guest graphs build a {\em simple cycle}. We created 3 graphs of this type: {\em SC\_simpleCycle50}, 
{\em SC\_simpleCycle500}\/ and {\em SC\_simpleCycle1000}\/ with $50$, $500$ and $1000$ vertices,  respectively.
\end{itemize}

\subsection{Randomly generated test instances}
\label{random:subsect}
The guest  graphs of these instances are marked by the prefix {\em ``RG\_''}. This instances are generated in the same way as the instances 
{\em SC\_random25}, {\em SC\_random50}\/ and {\em SC\_random75}.  
All guest graphs in this class of instances have $500$ vertices  and the  pairs of vertices are present as edges in the graphs randomly 
and independently  with the same  constant  probability, say  $\frac{x}{100}$.  
For each $x$ two  random  graphs are constructed as above and are denoted  by {\em RG\_randomAx}\/ and {\em RG\_randomBx}.
The degree of the regularity of the host tree is set to  $d = 2$ and $d = 7$.

\subsection{Instances with graphs taken from  Petit~\cite{Pet03}}
The guest graphs  of these instances are marked by the prefix {\em ``Pet03\_''}. 
These graphs were used in~\cite{Pet03} to test some heuristics for the linear arrangement problem (LAP), 
a problem related to the DAPT as explained in Section~\ref{intro:sec}. 
Also in this family of instances we use $d = 2$ and $d = 7$. This choice of the parameter $d$ is 
 motivated by the goal of  comparing  the behaviour of the proposed heuristics when a smaller and a larger value 
of the parameter $d$ are considered ($d=2$ and $d=7$).  

\section{Numerical results}
\label{numeric:sec}
The results of all numerical tests are summarised in the tables in Appendix. 
We group the test instances described in Section~\ref{TestInst:sec} in three groups: instances solvable by complete enumeration, polomially solvable instances and the rest. 
Table~\ref{tab:summaryForTheInstancesSolvedByTheCompleteEnumeration} reports
 on instances which could be solved to optimality by  complete enumeration 
on the following computer in $1$ week: HP Compaq nx7400,  32 bit Intel processor (Intel\textregistered Centrino\textsuperscript{\textregistered} Duo T2250 1.73 GHz), running in Ubuntu (Linux). 
Tabel~\ref{tab:summaryForTheInstancesSolvedByAPolynomialAlgorithm} summarises the results for the instances which are solvable to optimality in  polynomial time. 
Table~\ref{tab:summaryForTheInstancesWithoutAKnownOptimum} summarises the computational results obtained for the remaining instances.

In order to compare the quality of the proposed heuristics we define a {\sl quality quotient}\/ as follows 
\begin{equation}
	\label{equation:q}
	q(\mathscr{I}, \mathscr{H}) = \frac{1}{|\mathscr{I}|} 
\sum_{DAPT(G, d) \in \mathscr{I}}{\frac{\min_{HE \in \mathscr{H}}{\{HE(G, d)\}}}{\max{\{OS(G, d), DG(G, d)\}}}},
\end{equation}
where $\mathscr{I}$ denotes a  set of test instances $DAPT(G, d)$ with guest graph $G$ and degree of regularity $d$. $\mathscr{H}$ denotes a  
set of heuristics, $HE(G, d)$ stays for the  objective value obtained from the heuristic $HE$ for the instance $DAPT(G, d)$, $DG(G, d)$ represents
 the  degree bound for this instance and $OS(G, d)$ stays for the objective function value of an  optimal solution. 
We set $OS(G, d) = 0$ if the objective value of an optimal solution is unknown. We also write $q(G, d, \mathscr{H})$ for 
$q(\mathscr{I}, \mathscr{H})$ if $\mathscr{I} = \{DAPT(G, d)\}$.

 We evaluate also the so-called  {\em success factor}\/ which for a certain group  of instances and a certain heuristic gives 
the proportion of instances for which the considered heuristic computes the best known solution.

\subsection{Results on test instances solvable by complete enumeration}
Let us first consider Table~\ref{tab:summaryForTheInstancesSolvedByTheCompleteEnumeration}. 
All instances of this class   are very small (in fact, they have only $5$ -- $10$ vertices), and thus most of the heuristics 
were able to return an optimal solution for many instances.   
The success factors for the instances of this group  are summarised in 
Figure~\ref{fig:successFactorOfTheLowerBoundAndOfTheHeuristicsForTheInstancesSolvedByTheCompleteEnumeration} 
(the  acronyms are listed on the last page). 
The {\em DB}\/ entry shows us the proportion of the test instances whose optimal objective function value  equals  the  degree bound.  
The degree bound coincides with the optimal objective function value  only in the special case $d = n$.  Notice that in this 
case all arrangements yield the same objective function value and the DAPT is trivial. 

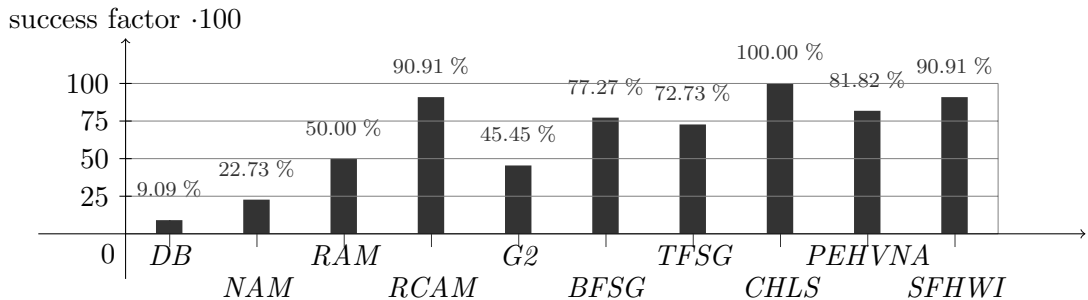
\begin{figure}[htbp]
	\begin{center}
		\begin{tikzpicture}[ycomb, xscale=1.16, yscale=0.02]
				
			\draw[color=black!80, line width=10pt]
				plot coordinates{(0.5000, 0.0000)(0.5000, 9.0909)} node[above] {\scriptsize $9.09 \; \%$};

			\draw[color=black!80, line width=10pt]
				plot coordinates{(1.5000, 0.0000)(1.5000, 22.7273)} node[above] {\scriptsize $22.73 \; \%$};

			\draw[color=black!80, line width=10pt]
				plot coordinates{(2.5000, 0.0000)(2.5000, 50.0000)} node[above] {\scriptsize $50.00 \; \%$};

			\draw[color=black!80, line width=10pt]
				plot coordinates{(3.5000, 0.0000)(3.5000, 90.9091)} node[above] {\scriptsize $90.91 \; \%$};

			\draw[color=black!80, line width=10pt]
				plot coordinates{(4.5000, 0.0000)(4.5000, 45.4545)} node[above] {\scriptsize $45.45 \; \%$};

			\draw[color=black!80, line width=10pt]
				plot coordinates{(5.5000, 0.0000)(5.5000, 77.2727)} node[above] {\scriptsize $77.27 \; \%$};

			\draw[color=black!80, line width=10pt]
				plot coordinates{(6.5000, 0.0000)(6.5000, 72.7273)} node[above] {\scriptsize $72.73 \; \%$};

			\draw[color=black!80, line width=10pt]
				plot coordinates{(7.5000, 0.0000)(7.5000, 100.0000)} node[above] {\scriptsize $100.00 \; \%$};

			\draw[color=black!80, line width=10pt]
				plot coordinates{(8.5000, 0.0000)(8.5000, 81.8182)} node[above] {\scriptsize $81.82 \; \%$};

			\draw[color=black!80, line width=10pt]
				plot coordinates{(9.5000, 0.0000)(9.5000, 90.9091)} node[above] {\scriptsize $90.91 \; \%$};

			\draw[very thin, color=gray, ystep=25, xstep=10] (0, 0) grid (10, 100);
			
			\draw[->] (-1, 0) -- (11, 0) node[right] {};
			\draw[->] (0, -30) -- (0, 130) node[above] {success factor $\cdot 100$};
			
			\foreach \pos in {0.5} \draw[shift={(\pos,-400pt)}] node {\it DB};
			\foreach \pos in {1.5} \draw[shift={(\pos,-1000pt)}] node {\it NAM};
			\foreach \pos in {2.5} \draw[shift={(\pos,-400pt)}] node {\it RAM};
			\foreach \pos in {3.5} \draw[shift={(\pos,-1000pt)}] node {\it RCAM};
			\foreach \pos in {4.5} \draw[shift={(\pos,-400pt)}] node {\it G2};
			\foreach \pos in {5.5} \draw[shift={(\pos,-1000pt)}] node {\it BFSG};
			\foreach \pos in {6.5} \draw[shift={(\pos,-400pt)}] node {\it TFSG};
			\foreach \pos in {7.5} \draw[shift={(\pos,-1000pt)}] node {\it CHLS};
			\foreach \pos in {8.5} \draw[shift={(\pos,-400pt)}] node {\it PEHVNA};
			\foreach \pos in {9.5} \draw[shift={(\pos,-1000pt)}] node {\it SFHWI};
			
			\foreach \pos in {0} \draw[shift={(\pos,0)}] node[below left] {$\pos$};
			\foreach \pos in {25, 50, 75, 100} \draw[shift={(0,\pos)}] (2pt,0pt) -- (-2pt,0pt) node[left] {$\pos$};
		\end{tikzpicture}
	\end{center}
	\caption{\em Success factors for the instances solved by  complete enumeration.}
	\label{fig:successFactorOfTheLowerBoundAndOfTheHeuristicsForTheInstancesSolvedByTheCompleteEnumeration}
\end{figure}

\subsection{Results on test instances solvable in polynomial time}
 Table~\ref{tab:summaryForTheInstancesSolvedByAPolynomialAlgorithm} is related to  the instances which can be solved by a polynomial time algorithm. This group of instances is  divided into four  parts as follows.

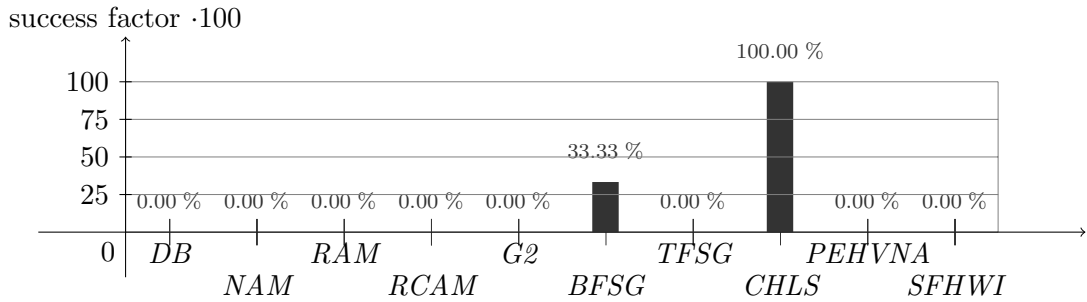
\begin{figure}[htbp]
	\begin{center}
		\begin{tikzpicture}[ycomb, xscale=1.16, yscale=0.02]
				
			\draw[color=black!80, line width=10pt]
				plot coordinates{(0.5000, 0.0000)(0.5000, 0.0000)} node[above] {\scriptsize $0.00 \; \%$};

			\draw[color=black!80, line width=10pt]
				plot coordinates{(1.5000, 0.0000)(1.5000, 0.0000)} node[above] {\scriptsize $0.00 \; \%$};

			\draw[color=black!80, line width=10pt]
				plot coordinates{(2.5000, 0.0000)(2.5000, 0.0000)} node[above] {\scriptsize $0.00 \; \%$};

			\draw[color=black!80, line width=10pt]
				plot coordinates{(3.5000, 0.0000)(3.5000, 0.0000)} node[above] {\scriptsize $0.00 \; \%$};

			\draw[color=black!80, line width=10pt]
				plot coordinates{(4.5000, 0.0000)(4.5000, 0.0000)} node[above] {\scriptsize $0.00 \; \%$};

			\draw[color=black!80, line width=10pt]
				plot coordinates{(5.5000, 0.0000)(5.5000, 33.3333)} node[above] {\scriptsize $33.33 \; \%$};

			\draw[color=black!80, line width=10pt]
				plot coordinates{(6.5000, 0.0000)(6.5000, 0.0000)} node[above] {\scriptsize $0.00 \; \%$};

			\draw[color=black!80, line width=10pt]
				plot coordinates{(7.5000, 0.0000)(7.5000, 100.0000)} node[above] {\scriptsize $100.00 \; \%$};

			\draw[color=black!80, line width=10pt]
				plot coordinates{(8.5000, 0.0000)(8.5000, 0.0000)} node[above] {\scriptsize $0.00 \; \%$};

			\draw[color=black!80, line width=10pt]
				plot coordinates{(9.5000, 0.0000)(9.5000, 0.0000)} node[above] {\scriptsize $0.00 \; \%$};

			\draw[very thin, color=gray, ystep=25, xstep=10] (0, 0) grid (10, 100);
			
			\draw[->] (-1, 0) -- (11, 0) node[right] {};
			\draw[->] (0, -30) -- (0, 130) node[above] {success factor $\cdot 100$};
			
			\foreach \pos in {0.5} \draw[shift={(\pos,-400pt)}] node {\it DB};
			\foreach \pos in {1.5} \draw[shift={(\pos,-1000pt)}] node {\it NAM};
			\foreach \pos in {2.5} \draw[shift={(\pos,-400pt)}] node {\it RAM};
			\foreach \pos in {3.5} \draw[shift={(\pos,-1000pt)}] node {\it RCAM};
			\foreach \pos in {4.5} \draw[shift={(\pos,-400pt)}] node {\it G2};
			\foreach \pos in {5.5} \draw[shift={(\pos,-1000pt)}] node {\it BFSG};
			\foreach \pos in {6.5} \draw[shift={(\pos,-400pt)}] node {\it TFSG};
			\foreach \pos in {7.5} \draw[shift={(\pos,-1000pt)}] node {\it CHLS};
			\foreach \pos in {8.5} \draw[shift={(\pos,-400pt)}] node {\it PEHVNA};
			\foreach \pos in {9.5} \draw[shift={(\pos,-1000pt)}] node {\it SFHWI};
			
			\foreach \pos in {0} \draw[shift={(\pos,0)}] node[below left] {$\pos$};
			\foreach \pos in {25, 50, 75, 100} \draw[shift={(0,\pos)}] (2pt,0pt) -- (-2pt,0pt) node[left] {$\pos$};
		\end{tikzpicture}
	\end{center}
	\caption{\em Success factors  for the instances with $d = n - 1$.}
	\label{fig:successFactorOfTheLowerBoundAndOfTheHeuristicsForTheInstancesWhichHoldTheEqualityDEqualNMinus1}
\end{figure}

\begin{figure}[htbp]
	\begin{center}
		\begin{tikzpicture}[ycomb, xscale=1.16, yscale=0.02]
				
			\draw[color=black!80, line width=10pt]
				plot coordinates{(0.5000, 0.0000)(0.5000, 0.0000)} node[above] {\scriptsize $0.00 \; \%$};

			\draw[color=black!80, line width=10pt]
				plot coordinates{(1.5000, 0.0000)(1.5000, 100.0000)} node[above] {\scriptsize $100.00 \; \%$};

			\draw[color=black!80, line width=10pt]
				plot coordinates{(2.5000, 0.0000)(2.5000, 0.0000)} node[above] {\scriptsize $0.00 \; \%$};

			\draw[color=black!80, line width=10pt]
				plot coordinates{(3.5000, 0.0000)(3.5000, 100.0000)} node[above] {\scriptsize $100.00 \; \%$};

			\draw[color=black!80, line width=10pt]
				plot coordinates{(4.5000, 0.0000)(4.5000, 0.0000)} node[above] {\scriptsize $0.00 \; \%$};

			\draw[color=black!80, line width=10pt]
				plot coordinates{(5.5000, 0.0000)(5.5000, 0.0000)} node[above] {\scriptsize $0.00 \; \%$};

			\draw[color=black!80, line width=10pt]
				plot coordinates{(6.5000, 0.0000)(6.5000, 100.0000)} node[above] {\scriptsize $100.00 \; \%$};

			\draw[color=black!80, line width=10pt]
				plot coordinates{(7.5000, 0.0000)(7.5000, 100.0000)} node[above] {\scriptsize $100.00 \; \%$};

			\draw[color=black!80, line width=10pt]
				plot coordinates{(8.5000, 0.0000)(8.5000, 100.0000)} node[above] {\scriptsize $100.00 \; \%$};

			\draw[color=black!80, line width=10pt]
				plot coordinates{(9.5000, 0.0000)(9.5000, 100.0000)} node[above] {\scriptsize $100.00 \; \%$};

			\draw[very thin, color=gray, ystep=25, xstep=10] (0, 0) grid (10, 100);
			
			\draw[->] (-1, 0) -- (11, 0) node[right] {};
			\draw[->] (0, -30) -- (0, 130) node[above] {success factor $\cdot 100$};
			
			\foreach \pos in {0.5} \draw[shift={(\pos,-400pt)}] node {\it DB};
			\foreach \pos in {1.5} \draw[shift={(\pos,-1000pt)}] node {\it NAM};
			\foreach \pos in {2.5} \draw[shift={(\pos,-400pt)}] node {\it RAM};
			\foreach \pos in {3.5} \draw[shift={(\pos,-1000pt)}] node {\it RCAM};
			\foreach \pos in {4.5} \draw[shift={(\pos,-400pt)}] node {\it G2};
			\foreach \pos in {5.5} \draw[shift={(\pos,-1000pt)}] node {\it BFSG};
			\foreach \pos in {6.5} \draw[shift={(\pos,-400pt)}] node {\it TFSG};
			\foreach \pos in {7.5} \draw[shift={(\pos,-1000pt)}] node {\it CHLS};
			\foreach \pos in {8.5} \draw[shift={(\pos,-400pt)}] node {\it PEHVNA};
			\foreach \pos in {9.5} \draw[shift={(\pos,-1000pt)}] node {\it SFHWI};
			
			\foreach \pos in {0} \draw[shift={(\pos,0)}] node[below left] {$\pos$};
			\foreach \pos in {25, 50, 75, 100} \draw[shift={(0,\pos)}] (2pt,0pt) -- (-2pt,0pt) node[left] {$\pos$};
		\end{tikzpicture}
	\end{center}
	\caption{\em Success factors  for the instance with $d = 4$ and  the guest  graph of Figure~\ref{fig:noMinimumContinuousArrangementExistsGraph}.}
	\label{fig:successFactorOfTheLowerBoundAndOfTheHeuristicsForTheInstancesConsistingOfTheGraphInFigure3AndOfTheChoiceDEqual4}
\end{figure}
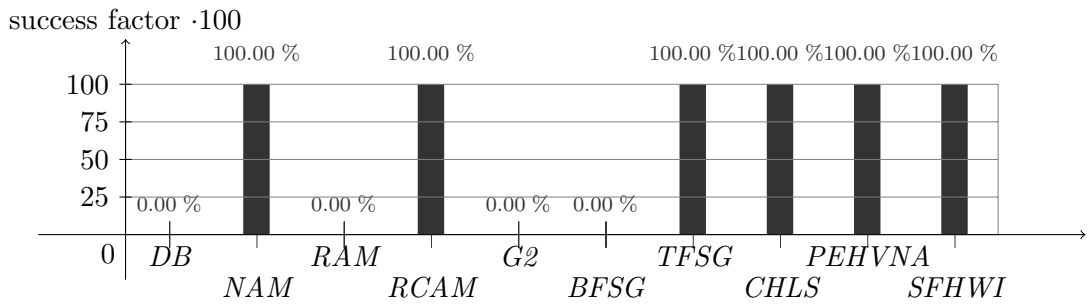

\begin{figure}[htbp]
	\begin{center}
		\begin{tikzpicture}[ycomb, xscale=1.16, yscale=0.02]
				
			\draw[color=black!80, line width=10pt]
				plot coordinates{(0.5000, 0.0000)(0.5000, 0.0000)} node[above] {\scriptsize $0.00 \; \%$};

			\draw[color=black!80, line width=10pt]
				plot coordinates{(1.5000, 0.0000)(1.5000, 0.0000)} node[above] {\scriptsize $0.00 \; \%$};

			\draw[color=black!80, line width=10pt]
				plot coordinates{(2.5000, 0.0000)(2.5000, 0.0000)} node[above] {\scriptsize $0.00 \; \%$};

			\draw[color=black!80, line width=10pt]
				plot coordinates{(3.5000, 0.0000)(3.5000, 0.0000)} node[above] {\scriptsize $0.00 \; \%$};

			\draw[color=black!80, line width=10pt]
				plot coordinates{(4.5000, 0.0000)(4.5000, 0.0000)} node[above] {\scriptsize $0.00 \; \%$};

			\draw[color=black!80, line width=10pt]
				plot coordinates{(5.5000, 0.0000)(5.5000, 0.0000)} node[above] {\scriptsize $0.00 \; \%$};

			\draw[color=black!80, line width=10pt]
				plot coordinates{(6.5000, 0.0000)(6.5000, 80.0000)} node[above] {\scriptsize $80.00 \; \%$};

			\draw[color=black!80, line width=10pt]
				plot coordinates{(7.5000, 0.0000)(7.5000, 10.0000)} node[above] {\scriptsize $10.00 \; \%$};

			\draw[color=black!80, line width=10pt]
				plot coordinates{(8.5000, 0.0000)(8.5000, 10.0000)} node[above] {\scriptsize $10.00 \; \%$};

			\draw[color=black!80, line width=10pt]
				plot coordinates{(9.5000, 0.0000)(9.5000, 0.0000)} node[above] {\scriptsize $0.00 \; \%$};

			\draw[very thin, color=gray, ystep=25, xstep=10] (0, 0) grid (10, 100);
			
			\draw[->] (-1, 0) -- (11, 0) node[right] {};
			\draw[->] (0, -30) -- (0, 130) node[above] {success factor $\cdot 100$};
			
			\foreach \pos in {0.5} \draw[shift={(\pos,-400pt)}] node {\it DB};
			\foreach \pos in {1.5} \draw[shift={(\pos,-1000pt)}] node {\it NAM};
			\foreach \pos in {2.5} \draw[shift={(\pos,-400pt)}] node {\it RAM};
			\foreach \pos in {3.5} \draw[shift={(\pos,-1000pt)}] node {\it RCAM};
			\foreach \pos in {4.5} \draw[shift={(\pos,-400pt)}] node {\it G2};
			\foreach \pos in {5.5} \draw[shift={(\pos,-1000pt)}] node {\it BFSG};
			\foreach \pos in {6.5} \draw[shift={(\pos,-400pt)}] node {\it TFSG};
			\foreach \pos in {7.5} \draw[shift={(\pos,-1000pt)}] node {\it CHLS};
			\foreach \pos in {8.5} \draw[shift={(\pos,-400pt)}] node {\it PEHVNA};
			\foreach \pos in {9.5} \draw[shift={(\pos,-1000pt)}] node {\it SFHWI};
			
			\foreach \pos in {0} \draw[shift={(\pos,0)}] node[below left] {$\pos$};
			\foreach \pos in {25, 50, 75, 100} \draw[shift={(0,\pos)}] (2pt,0pt) -- (-2pt,0pt) node[left] {$\pos$};
		\end{tikzpicture}
	\end{center}
	\caption{\em Success factors  for the instances with a  $d$-regular tree as a guest graph.}
	\label{fig:successFactorOfTheLowerBoundAndOfTheHeuristicsForTheInstanceWhoseGraphsHaveTheFormOfADRegularTree}
\end{figure}
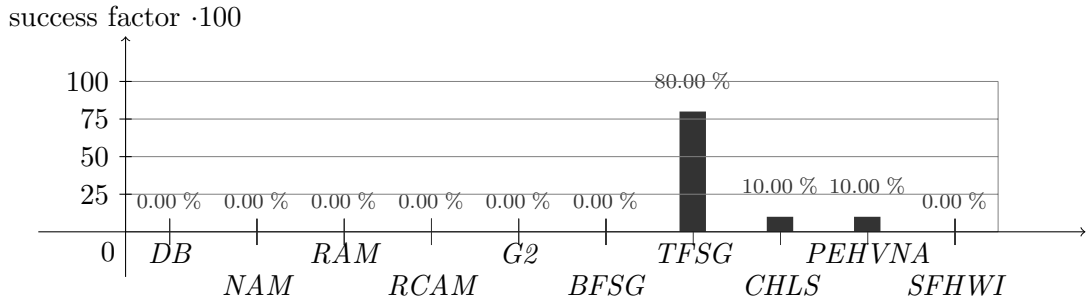

\begin{itemize}
	\item[(i)] Instances for which  the equality $d = n - 1$ holds.  The corresponding  success factors are given  in 
Figure~\ref{fig:successFactorOfTheLowerBoundAndOfTheHeuristicsForTheInstancesWhichHoldTheEqualityDEqualNMinus1}. 
It is interesting that  {\em CHLS}\/ yields the optimal solution for any  instance of this group, 
which  of course does not hold for all such DAPT  instances in general, cf.\ e.g.\ 
Figure~\ref{fig:successFactorOfTheLowerBoundAndOfTheHeuristicsForTheInstanceWhoseGraphsHaveTheFormOfADRegularTree}.
	\item[(ii)] Instances whose guest graph is a {\em star},  a {\em simple path}\/ or a {\em simple cycle}. 
Most of the heuristics return an  optimal solution.
	\item[(iii)] The instance with the guest graph of  Figure~\ref{fig:noMinimumContinuousArrangementExistsGraph} and $d = 4$. 
The corresponding  success factors are given  in 
Figure~\ref{fig:successFactorOfTheLowerBoundAndOfTheHeuristicsForTheInstancesConsistingOfTheGraphInFigure3AndOfTheChoiceDEqual4}. 
Note that neither the lower bound nor any heuristics is able to reach the optimum. 
Note also that some heuristics can generate a non-continuous arrangement. In our implementation they are RAM, CHLS and SFHWI.
	\item[(iv)] Instances with  a $d${\em-regular tree}\/ as a guest graph. 
No heuristic is able to return an  optimal arrangement for these instances and   {\em TFSG}\/ performs  mostly better than {\em CHLS}.
 The quality of the solutions is $q(\mathscr{I}, \mathscr{H}) \approx 1.18$. The corresponding   success factors are given in 
Figure~\ref{fig:successFactorOfTheLowerBoundAndOfTheHeuristicsForTheInstanceWhoseGraphsHaveTheFormOfADRegularTree}.
\end{itemize}

\subsection{Results on test instances with unknown  optimal solution}
Let us now consider the test instances with unknown  optimal solution, i.e.\ the optimal  solution of this instances  is not obtained by complete enumeration and it is not known whether it can be computed in polynomial time,  see Table~\ref{tab:summaryForTheInstancesWithoutAKnownOptimum}. The corresponding  success factors are given  in Figure~\ref{fig:successFactorOfTheLowerBoundAndOfTheHeuristicsForTheInstanceWithoutAKnownOptimum}.

\begin{figure}[htbp]
	\begin{center}
		\begin{tikzpicture}[ycomb, xscale=1.16, yscale=0.02]
				
			\draw[color=black!80, line width=10pt]
				plot coordinates{(0.5000, 0.0000)(0.5000, 0.0000)} node[above] {\scriptsize $0.00 \; \%$};

			\draw[color=black!80, line width=10pt]
				plot coordinates{(1.5000, 0.0000)(1.5000, 1.1905)} node[above] {\scriptsize $0.01 \; \%$};

			\draw[color=black!80, line width=10pt]
				plot coordinates{(2.5000, 0.0000)(2.5000, 0.0000)} node[above] {\scriptsize $0.00 \; \%$};

			\draw[color=black!80, line width=10pt]
				plot coordinates{(3.5000, 0.0000)(3.5000, 0.0000)} node[above] {\scriptsize $0.00 \; \%$};

			\draw[color=black!80, line width=10pt]
				plot coordinates{(4.5000, 0.0000)(4.5000, 1.1905)} node[above] {\scriptsize $0.01 \; \%$};

			\draw[color=black!80, line width=10pt]
				plot coordinates{(5.5000, 0.0000)(5.5000, 0.0000)} node[above] {\scriptsize $0.00 \; \%$};

			\draw[color=black!80, line width=10pt]
				plot coordinates{(6.5000, 0.0000)(6.5000, 1.1905)} node[above] {\scriptsize $0.01 \; \%$};

			\draw[color=black!80, line width=10pt]
				plot coordinates{(7.5000, 0.0000)(7.5000, 97.6190)} node[above] {\scriptsize $0.98 \; \%$};

			\draw[color=black!80, line width=10pt]
				plot coordinates{(8.5000, 0.0000)(8.5000, 2.3810)} node[above] {\scriptsize $0.02 \; \%$};

			\draw[color=black!80, line width=10pt]
				plot coordinates{(9.5000, 0.0000)(9.5000, 1.1905)} node[above] {\scriptsize $0.01 \; \%$};

			\draw[very thin, color=gray, ystep=25, xstep=10] (0, 0) grid (10, 100);
			
			\draw[->] (-1, 0) -- (11, 0) node[right] {};
			\draw[->] (0, -30) -- (0, 130) node[above] {success factor $\cdot 100$};
			
			\foreach \pos in {0.5} \draw[shift={(\pos,-400pt)}] node {\it DB};
			\foreach \pos in {1.5} \draw[shift={(\pos,-1000pt)}] node {\it NAM};
			\foreach \pos in {2.5} \draw[shift={(\pos,-400pt)}] node {\it RAM};
			\foreach \pos in {3.5} \draw[shift={(\pos,-1000pt)}] node {\it RCAM};
			\foreach \pos in {4.5} \draw[shift={(\pos,-400pt)}] node {\it G2};
			\foreach \pos in {5.5} \draw[shift={(\pos,-1000pt)}] node {\it BFSG};
			\foreach \pos in {6.5} \draw[shift={(\pos,-400pt)}] node {\it TFSG};
			\foreach \pos in {7.5} \draw[shift={(\pos,-1000pt)}] node {\it CHLS};
			\foreach \pos in {8.5} \draw[shift={(\pos,-400pt)}] node {\it PEHVNA};
			\foreach \pos in {9.5} \draw[shift={(\pos,-1000pt)}] node {\it SFHWI};
			
			\foreach \pos in {0} \draw[shift={(\pos,0)}] node[below left] {$\pos$};
			\foreach \pos in {25, 50, 75, 100} \draw[shift={(0,\pos)}] (2pt,0pt) -- (-2pt,0pt) node[left] {$\pos$};
		\end{tikzpicture}
	\end{center}
	\caption{\em Success factors  for the instances with an  unknown optimum.}
	\label{fig:successFactorOfTheLowerBoundAndOfTheHeuristicsForTheInstanceWithoutAKnownOptimum}
\end{figure}
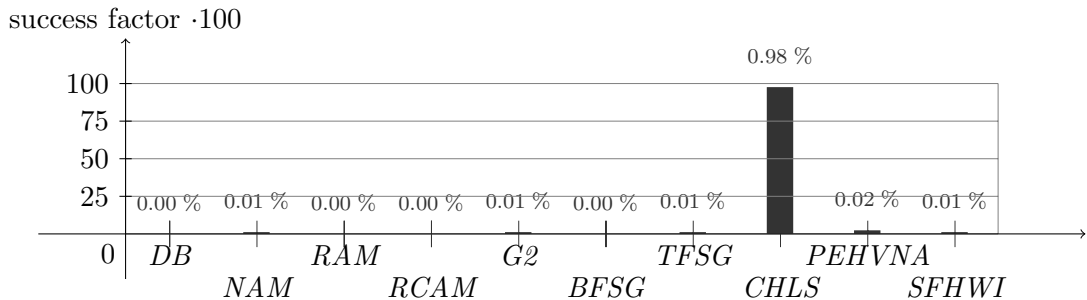

In the following we make some remarks on the particular classes of instances from this group.  
\smallskip

Consider first  the randomly generated instances (with  prefix {\em RG\_}). For all these instances  {\em CHLS}\/ outperforms  
the other heuristics.
 We  also observe  that for any fixed $d$  the quotients $q(\mathscr{I}, \mathscr{H})$ are better for denser graphs. 
The overall quality quotient for all  these instances is $q(\mathscr{I}, \mathscr{H}) \approx 1.21$. 
The quality quotient is better if $d = 2$; we get $q(\mathscr{I}, \mathscr{H}) \approx 1.20$ over the instances $\mathscr{I}$ with $d=2$  
and  $q(\mathscr{I}, \mathscr{H}) \approx 1.23$ for the other instances of this group. 

For the random instances ({\em RG\_randomAx}\/ and {\em RG\_randomBx}\/) we also observe an improvement of the quality quotient 
depending  on  the increasing expected density of  the guest  graph. 
Figure~\ref{fig:progressOfTheQualityDependingOnTheDensity} shows the values of  quality quotient computed for each pair of 
instances with guest graphs {\em RG\_randomAx}\/ and {\em RG\_randomBx}, for $x\in \{5,15,25,25,45,55,65,75,85,95\}$, and $d=2$ or $d=7$ 
respectively. Clearly $x$ represents the expected density of graphs generated as described in Section~\ref{random:subsect}. 

\medskip

Next  consider the instances with guest graphs taken  from  \textsc{Petit}~\cite{Pet03}. 
Let us notice that we have not considered  the guest  graphs  {\em Pet03\_crack}\/ with  $d = 2$ and have also excluded  {\em Pet03\_wave}\/ 
and {\em Pet03\_small}\/ as guest graphs from our tests.  
The reason is  the big size of  the guest graphs for the first two cases and the obtained solution by complete enumeration in the third case. 
The quality quotient  is  $q(\mathscr{I}, \mathscr{H}) \approx 1.84$ for this group of instances.
 For $d = 2$ we get $q(\mathscr{I}, \mathscr{H}) \approx 1.93$  and for $d = 7$ we get   $q(\mathscr{I}, \mathscr{H}) \approx 1.76$. 
Note that the quality quotient  is worse for these instances than for the {\em RG\_}\/ instances.

 A special behaviour could be observed on  following test instances:
\begin{itemize}
	\item The guest graph is given by  {\em Pet03\_hc10}\/ and $d = 2$. The underlying graph corresponds to a  
10-hypercube.
 Five  heuristics yield solutions with the same   objective function value which is the best know so far. 
It is worth of investigating  whether this  objective function value is optimal.

	\item The guest graph is given by {\em Pet03\_bintree10}\/ (a binary tree of height $10$) and $d=7$.
 This  problem  is polynomially solvable in the case that $d = 2$~\cite{CeSta13}. 
For $d \neq 2$ the computational complexity of this problem  is still open. 
We observe that  {\em TFSG}\/ performs  better than  {\em CHLS}\/ for both instances with the guest graph {\em Pet03\_bintree10}\/ 
and $d=7$ or $d=2$, respectively. 
\end{itemize}
\begin{figure}[htb]
	\begin{center}
		\begin{tikzpicture}[ycomb, xscale=0.1, yscale=4.4]
			\draw[very thin, color=gray, ystep=0.2, xstep=5] (0, 0) grid (100, 1.8);
			
			\draw[shift={(5, 1.791)}] (-10pt,-0.25pt) -- (10pt,0.25pt) (-10pt,0.25pt) -- (10pt,-0.25pt);
			\draw[shift={(15, 1.398)}] (-10pt,-0.25pt) -- (10pt,0.25pt) (-10pt,0.25pt) -- (10pt,-0.25pt);
			\draw[shift={(25, 1.270)}] (-10pt,-0.25pt) -- (10pt,0.25pt) (-10pt,0.25pt) -- (10pt,-0.25pt);
			\draw[shift={(35, 1.180)}] (-10pt,-0.25pt) -- (10pt,0.25pt) (-10pt,0.25pt) -- (10pt,-0.25pt);
			\draw[shift={(45, 1.134)}] (-10pt,-0.25pt) -- (10pt,0.25pt) (-10pt,0.25pt) -- (10pt,-0.25pt);
			\draw[shift={(55, 1.097)}] (-10pt,-0.25pt) -- (10pt,0.25pt) (-10pt,0.25pt) -- (10pt,-0.25pt);
			\draw[shift={(65, 1.060)}] (-10pt,-0.25pt) -- (10pt,0.25pt) (-10pt,0.25pt) -- (10pt,-0.25pt);
			\draw[shift={(75, 1.035)}] (-10pt,-0.25pt) -- (10pt,0.25pt) (-10pt,0.25pt) -- (10pt,-0.25pt);
			\draw[shift={(85, 1.018)}] (-10pt,-0.25pt) -- (10pt,0.25pt) (-10pt,0.25pt) -- (10pt,-0.25pt);
			\draw[shift={(95, 1.001)}] (-10pt,-0.25pt) -- (10pt,0.25pt) (-10pt,0.25pt) -- (10pt,-0.25pt);
		
			\draw[shift={(5, 1.689)}] (-10pt,-0.0pt) -- (10pt,0.0pt) (0pt,-0.25pt) -- (0pt,0.25pt);
			\draw[shift={(15, 1.378)}] (-10pt,-0.0pt) -- (10pt,0.0pt) (0pt,-0.25pt) -- (0pt,0.25pt);
			\draw[shift={(25, 1.245)}] (-10pt,-0.0pt) -- (10pt,0.0pt) (0pt,-0.25pt) -- (0pt,0.25pt);
			\draw[shift={(35, 1.200)}] (-10pt,-0.0pt) -- (10pt,0.0pt) (0pt,-0.25pt) -- (0pt,0.25pt);
			\draw[shift={(45, 1.177)}] (-10pt,-0.0pt) -- (10pt,0.0pt) (0pt,-0.25pt) -- (0pt,0.25pt);
			\draw[shift={(55, 1.164)}] (-10pt,-0.0pt) -- (10pt,0.0pt) (0pt,-0.25pt) -- (0pt,0.25pt);
			\draw[shift={(65, 1.156)}] (-10pt,-0.0pt) -- (10pt,0.0pt) (0pt,-0.25pt) -- (0pt,0.25pt);
			\draw[shift={(75, 1.119)}] (-10pt,-0.0pt) -- (10pt,0.0pt) (0pt,-0.25pt) -- (0pt,0.25pt);
			\draw[shift={(85, 1.076)}] (-10pt,-0.0pt) -- (10pt,0.0pt) (0pt,-0.25pt) -- (0pt,0.25pt);
			\draw[shift={(95, 1.047)}] (-10pt,-0.0pt) -- (10pt,0.0pt) (0pt,-0.25pt) -- (0pt,0.25pt);
			
			\draw[->] (-10, 0) -- (110, 0) node[right] {$\frac{2 m}{n (n - 1)} \cdot 100$};
			\draw[->] (0, -0.182) -- (0, 2.000) node[above] {$q(\mathscr{I}, \mathscr{H})$};
			
			\foreach \pos in {0} \draw[shift={(\pos,0)}] node[below left] {$\pos$};
			\foreach \pos in {5, 10, 15, 20, 25, 30, 35, 40, 45, 50, 55, 60, 65, 70, 75, 80, 85, 90, 95, 100} \draw[shift={(\pos, 0)}] (0pt,0.4pt) -- (0pt,-0.4pt) node[below] {$\pos$};
			\foreach \pos in {0.2, 0.4, 0.6, 0.8, 1.0, 1.2, 1.4, 1.6, 1.8} \draw[shift={(0,\pos)}] (17pt,0pt) -- (-17pt,0pt) node[left] {$\pos$};
			
			\draw[shift={(25.00, -0.40)}] (-10pt,-0.25pt) -- (10pt,0.25pt) (-10pt,0.25pt) -- (10pt,-0.25pt) node[right] {$d = 2$};
			\draw[shift={(75.00, -0.40)}] (-10pt,-0.0pt) -- (10pt,0.0pt) (0pt,-0.25pt) -- (0pt,0.25pt) node[right] {$d = 7$};
		\end{tikzpicture}
	\end{center}
	\caption{\em Progress of the quality quotient $q(\mathscr{I}, \mathscr{H})$ as a function of  the expected density of the randomly generated guest graphs.}
	\label{fig:progressOfTheQualityDependingOnTheDensity}
\end{figure}
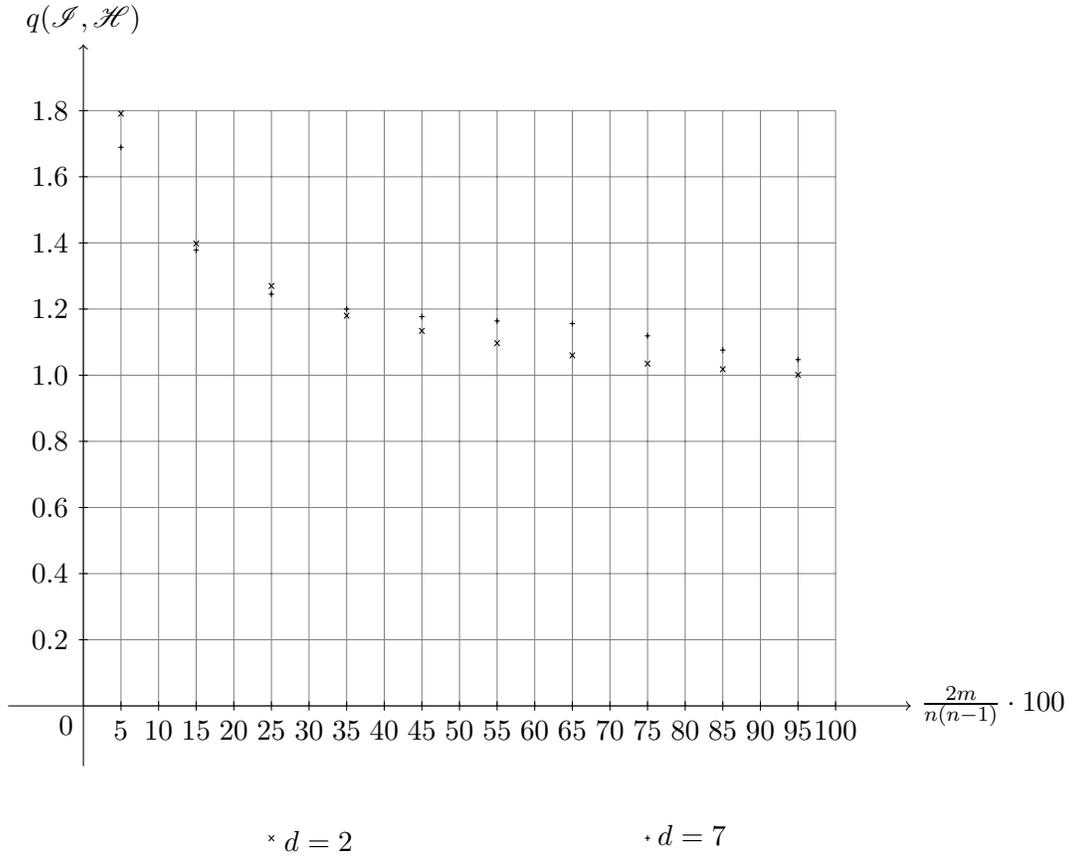

\subsection{Performance of the construction heuristic}\label{sub:construct}
In Tables~\ref{tab:summaryForTheInstancesSolvedByTheCompleteEnumeration}, \ref{tab:summaryForTheInstancesSolvedByAPolynomialAlgorithm} and 
\ref{tab:summaryForTheInstancesWithoutAKnownOptimum}  only the variant of the  construction heuristic which uses the simple  local search idea
  (see Section~\ref{subsub:MCBSSP}) to solve MCBSSP is included. 
This strategy outperforms the other one which uses the algorithm proposed by Feige, Krauthgamer and Nissim~\cite{Feige01} as a subroutine 
to solve MCBSSP. 
 Table~\ref{tab:ComparisonBetweenTwoStrategiesUsedToSolveMCSPBSS} provides some results on the comparison of the construction heuristic 
involving both approaches to solve MCBSSP, respectively. 
In this table there is only one instance  for which the involvement of the algorithm of Feige et al.\  yields better results. 
The guest graph of this instance is 2-regular  tree with $d=2$, hence this is  an instance of a special case of the DAPT 
solvable in polynomial time, see \cite{CeSta13}. 

\begin{center}
	\begin{longtable}{|l|r|r|r||r|r|r|}
				\hline
				\multicolumn{7}{|c|}{Table~\ref{tab:ComparisonBetweenTwoStrategiesUsedToSolveMCSPBSS} -- Two strategies in our construction heuristic}\\
				\hline
				{\bf Graph} & {\bf d} & {\bf n} & {\bf m} & {\bf OS} & {\bf LS} & {\bf FKN}\\
				\hline\hline
			\endfirsthead
				\hline
				\multicolumn{7}{|c|}{Table~\ref{tab:ComparisonBetweenTwoStrategiesUsedToSolveMCSPBSS} -- Two strategies in our construction heuristic}\\
				\hline
				{\bf Graph} & {\bf d} & {\bf n} & {\bf m} & {\bf OS} & {\bf LS} & {\bf FKN}\\
				\hline\hline
			\endhead
				\hline
				\multicolumn{7}{|c|}{follow-up on the next page $\dots$}\\
				\hline
			\endfoot
			\endlastfoot
SC\_random50 &
499
	& 500
		& 62468
			& {\bf 125374}
				& {\bf 125374}
					& {\bf 125374}\\
SC\_treeDG2H8 &
2
	& 511
		& 510
			& 2434
				& 3466
					& {\bf 3188}\\
SC\_treeDG3H6 &
3
	& 1093
		& 1092
			& 3926
				& {\bf 5042}
					& 5234\\
			\hline\hline
RG\_randomA5 &
2
	& 500
		& 6126
			& --
				& {\bf 86628}
					& 94660\\
RG\_randomA55 &
2
	& 500
		& 68320
			& --
				& {\bf 1074734}
					& 1091954\\
RG\_randomA95 &
2
	& 500
		& 118499
			& --
				& {\bf 1893354}
					& 1898482\\
			\hline
RG\_randomA5 &
7
	& 500
		& 6126
			& --
				& {\bf 36392}
					& 39344\\
RG\_randomA55 &
7
	& 500
		& 68320
			& --
				& {\bf 446038}
					& 452220\\
RG\_randomA95 &
7
	& 500
		& 118499
			& --
				& {\bf 785158}
					& 787360\\
			\hline\hline
Pet03\_randomA1 &
2
	& 1000
		& 4974
			& --
				& {\bf 71874}
					& 81264\\
Pet03\_hc10 &
2
	& 1024
		& 5120
			& --
				& {\bf 56320}
					& {\bf 56320}\\
Pet03\_c1y &
2
	& 828
		& 1749
			& --
				& {\bf 16884}
					& 21454\\
Pet03\_gd95c &
2
	& 62
		& 144
			& --
				& {\bf 866}
					& 1044\\
			\hline
Pet03\_randomA1 &
7
	& 1000
		& 4974
			& --
				& {\bf 29574}
					& 33980\\
Pet03\_hc10 &
7
	& 1024
		& 5120
			& --
				& {\bf 25892}
					& 27580\\
Pet03\_c1y &
7
	& 828
		& 1749
			& --
				& {\bf 7508}
					& 9016\\
Pet03\_gd95c &
7
	& 62
		& 144
			& --
				& {\bf 410}
					& 500\\
			\hline
		\multicolumn{7}{c}{\parbox{\LTcapwidth}{}}\\
		\caption{\em A comparison of the two approaches used to solve MCBSSP as a subroutine   in the  construction heuristic: the  local search idea (LS) and  the algorithm proposed by Feige et al.~\cite{Feige01} (FKN), see Section~\ref{subsub:MCBSSP}.}
		\label{tab:ComparisonBetweenTwoStrategiesUsedToSolveMCSPBSS}
	\end{longtable}
\end{center}

\section{Conclusions and outlook}
\label{sec:conclusionsAndOutlook}
In this paper we deal with the data arrangement problem on regular trees DAPT, 
identify some basic properties  and introduce heuristic approaches for this problem. 
We provide a comparative analysis of the proposed heuristics based on a set of test instances we have generated. 
To the best of our knowledge  no sources of literature dealing with heuristic approaches for the DAPT are available. 
So there is no possibility to test the performance of the proposed heuristics on already 
known benchmark instances and neither to compare the proposed heuristics to already existing approaches  in the literature. 
However we make use of test instances available in Petit~\cite{Pet03}  for a related 
problem, the linear arrangement problem,   and use these graphs as a guest graph in our test instances. We have summarised the generated test  instances in a   library which is available at  \url{http://www.opt.math.tu-graz.ac.at/~cela/public.htm}.

There is plenty of room for further research on this topic in the future. 
Most of the heuristics we propose are basis approaches which can be well combined with one another. Especially we expect a significant performance improvement if the two local search heuristics we propose are combined in order to may escape form the local minima of our neighbourhood by making a jump in the other neighbourhood. 
Also in the  construction heuristic there is room for improvement, especially as far as the subroutine used to solve MCBSSP is concerned. Since this problem has been investigated to some extent in the literature there is hope for appropriate approaches to make use of in the construction heuristic.  Another aspect which could be  considered is an alternative   handling of   the unused leaves. 
\section*{Acknowledgements}

The research was funded by the Austrian Science Fund (FWF): P23829.


\newgeometry{top=2.039cm, bottom=2.039cm}

\begin{landscape}
	\pagestyle{empty}
	\scriptsize

\section*{Appendix}
\label{sec:appendix}
	
	\begin{center}
		\begin{longtable}{|l|r|r|r||r|r|r|r|r|r|r|r|r|Er|r|}
				\hline
				\multicolumn{16}{|c|}{Table~\ref{tab:summaryForTheInstancesSolvedByTheCompleteEnumeration} -- summary for the instances solved by the complete enumeration}\\
				\hline
				{\bf Graph} & {\bf d} & {\bf n} & {\bf m} & {\bf OS} & {\bf DB} & {\bf NAM} & {\bf RAM} & {\bf RCAM} & {\bf G2} & {\bf BFSG} & {\bf TFSG} & {\bf CHLS} & {\bf CHFKN} & {\bf PEHVNA} & {\bf SFHWI}\\
				\hline\hline
			\endfirsthead
				\hline
				\multicolumn{16}{|c|}{Table~\ref{tab:summaryForTheInstancesSolvedByTheCompleteEnumeration} -- summary for the instances solved by the complete enumeration -- follow up}\\
				\hline
				{\bf Graph} & {\bf d} & {\bf n} & {\bf m} & {\bf OS} & {\bf DB} & {\bf NAM} & {\bf RAM} & {\bf RCAM} & {\bf G2} & {\bf BFSG} & {\bf TFSG} & {\bf CHLS} & {\bf CHFKN} & {\bf PEHVNA} & {\bf SFHWI}\\
				\hline\hline
			\endhead
				\hline
				\multicolumn{16}{|c|}{follow-up on the next page $\dots$}\\
				\hline
			\endfoot
			\endlastfoot
CE\_sample &
3
	& 5
		& 7
			& {\bf 20}
				& 18
					& 22
						& {\bf 20}
							& {\bf 20}
								& 22
									& {\bf 20}
										& {\bf 20}
											& {\bf 20}
												& 22
													& {\bf 20}
														& {\bf 20}\\
			\hline
CE\_thin7 &
2
	& 7
		& 2
			& {\bf 24}
				& 21
					& 26
						& {\bf 24}
							& {\bf 24}
								& 26
									& {\bf 24}
										& {\bf 24}
											& {\bf 24}
												& 26
													& 26
														& {\bf 24}\\
CE\_thin7 &
3
	& 7
		& 2
			& {\bf 18}
				& 16
					& 20
						& {\bf 18}
							& {\bf 18}
								& 20
									& {\bf 18}
										& {\bf 18}
											& {\bf 18}
												& {\bf 18}
													& {\bf 18}
														& {\bf 18}\\
CE\_thin7 &
4
	& 7
		& 2
			& {\bf 16}
				& 14
					& {\bf 16}
						& 18
							& {\bf 16}
								& {\bf 16}
									& {\bf 16}
										& {\bf 16}
											& {\bf 16}
												& {\bf 16}
													& {\bf 16}
														& {\bf 16}\\
CE\_thin7 &
5
	& 7
		& 2
			& {\bf 16}
				& 14
					& 18
						& 18
							& {\bf 16}
								& 18
									& 18
										& 18
											& {\bf 16}
												& {\bf 16}
													& 18
														& {\bf 16}\\
CE\_thin7 &
6
	& 7
		& 2
			& {\bf 16}
				& 14
					& 18
						& 18
							& {\bf 16}
								& 18
									& {\bf 16}
										& {\bf 16}
											& {\bf 16}
												& {\bf 16}
													& {\bf 16}
														& {\bf 16}\\
CE\_thin7 &
7
	& 7
		& 2
			& {\bf 14}
				& {\bf 14}
					& {\bf 14}
						& {\bf 14}
							& {\bf 14}
								& {\bf 14}
									&  {\bf 14}
										& {\bf 14}
											& {\bf 14}
												& {\bf 14}
													& {\bf 14}
														& {\bf 14}\\
			\hline
CE\_dense7 &
2
	& 7
		& 14
			& {\bf 62}
				& 56
					& 64
						& {\bf 62}
							& {\bf 62}
								& 64
									& {\bf 62}
										& 66
											& {\bf 62}
												& 64
													& {\bf 62}
														& {\bf 62}\\
CE\_dense7 &
3
	& 7
		& 14
			& {\bf 44}
				& 42
					& {\bf 44}
						& {\bf 44}
							& {\bf 44}
								& {\bf 44}
									& {\bf 44}
										& 46
											& {\bf 44}
												& 46
													& {\bf 44}
														& {\bf 44}\\
CE\_dense7 &
4
	& 7
		& 14
			& {\bf 40}
				& 35
					& 42
						& {\bf 40}
							& {\bf 40}
								& 42
									& {\bf 40}
										& 44
											& {\bf 40}
												& 42
													& {\bf 40}
														& {\bf 40}\\
CE\_dense7 &
5
	& 7	
		& 14	
			& {\bf 38}
				& 30
					& 40
						& 40
							& {\bf 38}
								& {\bf 38}
									& {\bf 38}
										& {\bf 38}
											& {\bf 38}
												& {\bf 38}
													& {\bf 38}
														& {\bf 38}\\
CE\_dense7 &
6
	& 7	
		& 14	
			& {\bf 34}
				& 28
					& 36
						& 38
							& {\bf 34}
								& {\bf 34}
									& {\bf 34}
										& {\bf 34}
											& {\bf 34}
												& {\bf 34}
													& {\bf 34}
														& {\bf 34}\\
CE\_dense7 &
7
	& 7
		& 14
			& {\bf 28}
				& {\bf 28}
					& {\bf 28}
						& {\bf 28}
							& {\bf 28}
								& {\bf 28}
									& {\bf 28}
										& {\bf 28}
											& {\bf 28}
												& {\bf 28}
													& {\bf 28}
														& {\bf 28}\\
			\hline
CE\_mesh9 &
2
	& 9
		& 12
			& {\bf 54}
				& 40
					& 58
						& 62
							& 56
								& 58
									& 62
										& {\bf 54}
											& {\bf 54}
												& {\bf 54}
													& {\bf 54}
														& {\bf 54}\\
CE\_mesh9 &
3
	& 9	
		& 12
			& {\bf 36}
				& 30
					& {\bf 36}
						& {\bf 36}
							& {\bf 36}
								& {\bf 36}
									& {\bf 36}
										& {\bf 36}
											& {\bf 36}
												& {\bf 36}
													& {\bf 36}
														& {\bf 36}\\
CE\_mesh9 &
4
	& 9	
		& 12
			& {\bf 34}
				& 25
					& 36
						& 36
							& {\bf 34}
								& 36
									& 36
										& {\bf 34}
											& {\bf 34}
												& {\bf 34}
													& {\bf 34}
														& {\bf 34}\\
			\hline
CE\_thin10 &
2
	& 10
		& 11
			& {\bf 46}
				& 34
					& 62
						& 54
							& 48
								& 60
									& 56
										& {\bf 46}
											& {\bf 46}
												& 56
													& 48
														& 48\\
CE\_thin10 &
4
	& 10
		& 11
			& {\bf 30}
				& 22
					& 34
						& 32
							& {\bf 30}
								& 34
									& 32
										& {\bf 30}
											& {\bf 30}
												& 32
													& {\bf 30}
														& {\bf 30}\\
			\hline
CE\_dense10 &
2
	& 10
		& 26
			& {\bf 134}
				& 118
					& 140
						& 150
							& {\bf 134}
								& 136
									& {\bf 134}
										& 140
											& {\bf 134}
												& 136
													& {\bf 134}
														& 136\\
CE\_dense10 &
4
	& 10
		& 26
			& {\bf 80}
				& 74
					& 82
						& 86
							& {\bf 80}
								& {\bf 80}
									& {\bf 80}
										& 86
											& {\bf 80}
												& 86
													& 82
														& {\bf 80}\\
			\hline
Pet03\_small &
2
	& 5
		& 8
			& {\bf 34}
				& 28
					& 36
						& {\bf 34}
							& {\bf 34}
								& {\bf 34}
									& {\bf 34}
										& {\bf 34}
											& {\bf 34}
												& {\bf 34}
													& {\bf 34}
														& {\bf 34}\\
Pet03\_small &
3
	& 5
		& 8
			& {\bf 24}
				& 22
					& 26
						& {\bf 24}
							& {\bf 24}
								& {\bf 24}
									& {\bf 24}
										& {\bf 24}
											& {\bf 24}
												& {\bf 24}
													& {\bf 24}
														& {\bf 24}\\
			\hline
			\multicolumn{16}{c}{\parbox{\LTcapwidth}{}}\\
			\caption{\em Summary for the instances solved by the complete enumeration.}
			\label{tab:summaryForTheInstancesSolvedByTheCompleteEnumeration}
		\end{longtable}
	\end{center}
	
	\vspace*{-0.7cm}

	\begin{center}
		\begin{longtable}{|l|r|r|r||r|r|r|r|r|r|r|r|r|Er|r|}
				\hline
				\multicolumn{16}{|c|}{Table~\ref{tab:summaryForTheInstancesSolvedByAPolynomialAlgorithm} -- summary for the instances solved by a polynomial time algorithm}\\
				\hline
				{\bf Graph} & {\bf d} & {\bf n} & {\bf m} & {\bf OS} & {\bf DB} & {\bf NAM} & {\bf RAM} & {\bf RCAM} & {\bf G2} & {\bf BFSG} & {\bf TFSG} & {\bf CHLS} & {\bf CHFKN} & {\bf PEHVNA} & {\bf SFHWI}\\
				\hline\hline
			\endfirsthead
				\hline
				\multicolumn{16}{|c|}{Table~\ref{tab:summaryForTheInstancesSolvedByAPolynomialAlgorithm} -- summary for the instances solved by a polynomial time algorithm -- follow up}\\
				\hline
				{\bf Graph} & {\bf d} & {\bf n} & {\bf m} & {\bf OS} & {\bf DB} & {\bf NAM} & {\bf RAM} & {\bf RCAM} & {\bf G2} & {\bf BFSG} & {\bf TFSG} & {\bf CHLS} & {\bf CHFKN} & {\bf PEHVNA} & {\bf SFHWI}\\
				\hline\hline
			\endhead
				\hline
				\multicolumn{16}{|c|}{follow-up on the next page $\dots$}\\
				\hline
			\endfoot
			\endlastfoot
SC\_random25 &
499
	& 500
		& 31239
			& {\bf 62644}
				& 62478
					& 62720
						& 124774
							& 62740
								& 62644
									& {\bf 62644}
										& 62704
											& {\bf 62644}
												& {\bf 62644}
													& 62706
														& 62706\\
SC\_random50 &
499
	& 500
		& 62468
			& {\bf 125374}
				& 124936
					& 125438
						& 249544
							& 125410
								& 125374
									& 125380
										& 125418
											& {\bf 125374}
												& {\bf 125374}
													& 125420
														& 125438\\
SC\_random75 &
499
	& 500
		& 93548
			& {\bf 187784}
				& 187096
					& 187832
						& 373724
							& 187822
								& 187784
									& 187796
										& 187840
											& {\bf 187784}
												& {\bf 187784}
													& 187824
														& 187832\\
			\hline											
SC\_star50 &
2
	& 50
		& 49
			& {\bf 474}
				& 286
					& {\bf 474}
						& 478
							& {\bf 474}
								& {\bf 474}
									& {\bf 474}
										& {\bf 474}
											& {\bf 474}
												& {\bf 474}
													& {\bf 474}
														& {\bf 474}\\
SC\_star500 &
2
	& 500
		& 499
			& {\bf 7978}
				& 4488
					& {\bf 7978}
						& 7980
							& {\bf 7978}
								& {\bf 7978}
									& {\bf 7978}
										& {\bf 7978}
											& {\bf 7978}
												& {\bf 7978}
													& {\bf 7978}
														& {\bf 7978}\\
SC\_star1000 &
2
	& 1000
		& 999
			& {\bf 17954}
				& 9976
					& {\bf 17954}
						& 17968
							& {\bf 17954}
								& {\bf 17954}
									& {\bf 17954}
										& {\bf 17954}
											& {\bf 17954}
												& {\bf }
													& {\bf 17954}
														& {\bf 17954}\\
			\hline											
SC\_star50 &
7
	& 50
		& 49
			& {\bf 186}
				& 142
					& {\bf 186}
						& 262
							& {\bf 186}
								& {\bf 186}
									& {\bf 186}
										& {\bf 186}
											& {\bf 186}
												& {\bf 186}
													& {\bf 186}
														& {\bf 186}\\
SC\_star500 &
7
	& 500
		& 499
			& {\bf 3200}
				& 2099
					& {\bf 3200}
						& 3770
							& {\bf 3200}
								& {\bf 3200}
									& {\bf 3200}
										& {\bf 3200}
											& {\bf 3200}
												& {\bf 3200}
													& {\bf 3200}
														& {\bf 3200}\\
SC\_star1000 &
7
	& 1000
		& 999
			& {\bf 7200}
				& 4599
					& {\bf 7200}
						& 7600
							& {\bf 7200}
								& {\bf 7200}
									& {\bf 7200}
										& {\bf 7200}
											& {\bf 7200}
												& {\bf 7200}
													& {\bf 7200}
														& {\bf 7200}\\
			\hline
SC\_extStar &
4
	& 12
		& 11
			& 28
				& 23
					& {\bf 30}
						& 32
							& {\bf 30}
								& 32
									& 32
										& {\bf 30}
											& {\bf 30}
												& {\bf 30}
													& {\bf 30}
														& {\bf 30}\\
			\hline
SC\_treeDG2H8 &
2
	& 511
		& 510
			& 2434
				& 1529
					& 8176
						& 7968
							& 7982
								& 8176
									& 7680
										& {\bf 2746}
											& 3466
												& 3188
													& 4878
														& 6426\\
Pet03\_bintree10 &
2
	& 1023
		& 1022
			& 4904
				& 3065
					& 18414
						& 18146
							& 18118
								& 18414
									& 17410
										& {\bf 5618}
											& 7072
												& 6566
													& 10696
														& 15030\\
SC\_treeDG2H10 &
2
	& 2047
		& 2046
			& 9850
				& 6137
					& 40940
						& 45508
							& 40542
								& 40940
									& 38914
										& {\bf 11418}
											& 16026
												& --
													& 23566
														& 34938\\
SC\_treeDG2H11 &
2
	& 4095
		& 4094
			& 19744
				& 12281
					& 90090
						& 89568
							& 89528
								& 90090
									& 86020
										& {\bf 23154}
											& 33748
												& --
													& 50826
														& 80064\\
SC\_treeDG2H12 &
2
	& 8191
		& 8190
			& 39538
				& 24569
					& 196584
						& 195714
							& 195668
								& 196584
									& 188420
										& {\bf 46930}
											& 71666
												& --
													& 109504
														& 174972\\
SC\_treeDG3H5 &
3
	& 364
		& 363
			& 1296
				& 967
					& 3640
						& 3870
							& 3642
								& 3640
									& 3640
										& 1524
											& {\bf 1472}
												& 1646
													& 2376
														& 2778\\
SC\_treeDG3H6 &
3
	& 1093
		& 1092
			& 3926
				& 2911
					& 13116
						& 14040
							& 13214
								& 13116
									& 13116
										& {\bf 4772}
											& 5042
												& 5234
													& 8380
														& 11314\\
SC\_treeDG4H4 &
4
	& 341
		& 340
			& 1058
				& 849
					& 2728
						& 3102
							& 2738
								& 2728
									& 2728
										& {\bf 1288}
											& 1328
												& 1346
													& 1584
														& 2144\\
SC\_treeDG4H5 &
4
	& 1365
		& 1364
			& 4272
				& 3409
					& 1365
						& 15310
							& 13884
								& 13650
									& 13650
										& {\bf 5320}
											& 5412
												&
													& 7626
														& 11820\\
SC\_treeDG8H3 &
8
	& 585
		& 584
			& 1472
				& 1313
					& 3510
						& 4426
							& 3526
								& 3510
									& 3510
										& 1956
											& 1808
												&
													& {\bf 1716}
														& 3018\\
			\hline
SC\_path50 &
2
	& 50
		& 49
			& {\bf 190}
				& 146
					& {\bf 190}
						& 434
							& 434
								& {\bf 190}
									& {\bf 190}
										& {\bf 190}
											& {\bf 190}
												& {\bf 190}
													& {\bf 190}
														& {\bf 190}\\
SC\_path500 &
2
	& 500
		& 499
			& {\bf 1982}
				& 1496
					& {\bf 1982}
						& 7818
							& 7814
								& {\bf 1982}
									& {\bf 1982}
										& {\bf 1982}
											& {\bf 1982}
												& {\bf 1982}
													& {\bf 1982}
														& {\bf 1982}\\
SC\_path1000 &
2
	& 1000
		& 999
			& {\bf 3980}
				& 2996
					& {\bf 3980}
						& 17706
							& 17726
								& {\bf 3980}
									& {\bf 3980}
										& {\bf 3980}
											& {\bf 3980}
												& {\bf 3980}
													& {\bf 3980}
														& {\bf 3980}\\
			\hline
SC\_path50 &
7
	& 50
		& 49
			& {\bf 114}
				& 98
					& {\bf 114}
						& 258
							& 180
								& {\bf 114}
									& {\bf 114}
										& {\bf 114}
											& {\bf 114}
												& {\bf 114}
													& {\bf 114}
														& {\bf 114}\\
SC\_path500 &
7
	& 500
		& 499
			& {\bf 1162}
				& 998
					& {\bf 1162}
						& 3754
							& 3260
								& {\bf 1162}
									& {\bf 1162}
										& {\bf 1162}
											& {\bf 1162}
												& {\bf 1162}
													& {\bf 1162}
														& {\bf 1162}\\
SC\_path1000 &
7
	& 1000
		& 999
			& {\bf 2326}
				& 1998
					& {\bf 2326}
						& 7562
							& 7114
								& {\bf 2326}
									& {\bf 2326}
										& {\bf 2326}
											& {\bf 2326}
												& {\bf 2326}
													& {\bf 2326}
														& {\bf 2326}\\
			\hline
SC\_simpleCycle50 &
2
	& 50
		& 50
			& {\bf 202}
				& 150
					& {\bf 202}
						& 436
							& 452
								& 284
									& 284
										& {\bf 202}
											& {\bf 202}
												& {\bf 202}
													& {\bf 202}
														& {\bf 202}\\
SC\_simpleCycle500 &
2
	& 500
		& 500
			& {\bf 2000}
				& 1500
					& {\bf 2000}
						& 7782
							& 7804
								& 2968
									& 2968
										& {\bf 2000}
											& {\bf 2000}
												& {\bf 2000}
													& {\bf 2000}
														& {\bf 2000}\\
SC\_simpleCycle1000 &
2
	& 1000
		& 1000
			& {\bf 4000}
				& 3000
					& {\bf 4000}
						& 17742
							& 17746
								& 5964
									& 5964
										& {\bf 4000}
											& {\bf 4000}
												& {\bf 4000}
													& {\bf 4000}
														& {\bf 4000}\\
			\hline
SC\_simpleCycle50 &
7
	& 50
		& 50
			& {\bf 120}
				& 100
					& {\bf 120}
						& 258
							& 178
								& 132
									& 132
										& {\bf 120}
											& {\bf 120}
												& {\bf 120}
													& {\bf 120}
														& {\bf 120}\\
SC\_simpleCycle500 &
7
	& 500
		& 500
			& {\bf 1170}
				& 1000
					& {\bf 1170}
						& 3770
							& 3256
								& 1328
									& 1328
										& {\bf 1170}
											& {\bf 1170}
												& {\bf 1170}
													& {\bf 1170}
														& {\bf 1170}\\
SC\_simpleCycle1000 &
7
	& 1000
		& 1000
			& {\bf 2334}
				& 2000
					& {\bf 2334}
						& 7576
							& 7128
								& 2656 
									& 2656
										& {\bf 2334}
											& {\bf 2334}
												& {\bf 2334}
													& {\bf 2334}
														& {\bf 2334}\\
			\hline
			\multicolumn{16}{c}{\parbox{\LTcapwidth}{}}\\
			\caption{\em Summary for the instances solved by a polynomial time algorithm.}
			\label{tab:summaryForTheInstancesSolvedByAPolynomialAlgorithm}
		\end{longtable}
	\end{center}
	
	\vspace*{-0.7cm}

	\begin{center}
		\begin{longtable}{|l|r|r|r||r|r|r|r|r|r|r|r|r|Er|r|}
				\hline
				\multicolumn{16}{|c|}{Table~\ref{tab:summaryForTheInstancesWithoutAKnownOptimum} -- summary for the instances without a known optimal solution} \\
				\hline
				{\bf Graph} & {\bf d} & {\bf n} & {\bf m} & {\bf OS} & {\bf DB} & {\bf NAM} & {\bf RAM} & {\bf RCAM} & {\bf G2} & {\bf BFSG} & {\bf TFSG} & {\bf CHLS} & {\bf CHFKN} & {\bf PEHVNA} & {\bf SFHWI}\\
				\hline\hline
			\endfirsthead
				\hline
				\multicolumn{16}{|c|}{Table~\ref{tab:summaryForTheInstancesWithoutAKnownOptimum} -- summary for the instances without a known optimal solution -- follow up} \\
				\hline
				{\bf Graph} & {\bf d} & {\bf n} & {\bf m} & {\bf OS} & {\bf DB} & {\bf NAM} & {\bf RAM} & {\bf RCAM} & {\bf G2} & {\bf BFSG} & {\bf TFSG} & {\bf CHLS} & {\bf CHFKN} & {\bf PEHVNA} & {\bf SFHWI}\\
				\hline\hline
			\endhead
				\hline
				\multicolumn{16}{|c|}{follow-up on the next page $\dots$} \\
				\hline
			\endfoot
			\endlastfoot
RG\_randomA5 &
2	& 500	& 6126	& --& 48361	& 98140	& 97640 	& 97566	& 92076	& 96816	& 98004	& {\bf 86628}	& 94660	& 89684	& 93230\\
RG\_randomB5 &
2	& 500	& 6175	& --& 48880	& 99288	& 98334		& 98344	& 92836	& 97556	& 98610	& {\bf 87540}	& 95534	& 90562	& 94990\\
RG\_randomA15 &
2	& 500	& 18654	& --& 201234& 299322& 298002	& 297740& 290074& 297304& 298738& {\bf 281686}	& 295470& 286346& 292550\\
RG\_randomB15 &
2	& 500	& 18887	& --& 204529& 302016& 301636	& 301704& 293428& 300738& 302536& {\bf 285518}	& 298684& 290026& 295792\\
RG\_randomA25 &
2	& 500	& 31254	& --& 379064& 500784& 499788	& 499738& 490796& 499032& 500724& {\bf 480966}	& 497080& 486306& 495964\\
RG\_randomB25 &
2	& 500	& 31114	& --& 376931& 497684& 497456	& 497524& 487788& 496732& 498384& {\bf 478812}	& 494338& 483780& 493214\\
RG\_randomA35 &
2	& 500	& 43605	& --& 574180& 699352& 697740	& 697594& 687438& 697148& 698420& {\bf 677870}	& 695098& 683514& 691738\\
RG\_randomB35 &
2	& 500	& 43595	& --& 574020& 699236& 697246	& 697540& 686808& 696756& 698494& {\bf 677294}	& 695248& 683192& 690294\\
RG\_randomA45 &
2	& 500	& 56653	& --& 782958& 908042& 907104	& 906882& 896474& 906354& 907852& {\bf 886786}	& 904142& 892358& 899844\\
RG\_randomB45 &
2	& 500	& 55627	& --& 766539& 891646& 890022	& 890094& 879572& 889732& 891178& {\bf 870416}	& 887984& 876214& 884740\\
RG\_randomA55 &
2	& 500	& 68320	& --& 978888& 1095540
										& 1093664	& 1093718
															& 1083122
																	& 1093128
																			& 1094468& {\bf 1074734}& 1091954& 1079810& 1090610\\
RG\_randomB55 &
2	& 500	& 68701	& --& 985749& 1101882
										& 1100048	& 1099958
															& 1089652
																	& 1099576
																			& 1101090& {\bf 1080722}& 1097676& 1085512& 1094074\\
RG\_randomA65 &
2	& 500	& 81279	& --& 1212022
								& 1302766
										& 1301848	& 1301356
															& 1291892
																	& 1300882
																			& 1302240& {\bf 1284086}& 1300032& 1288564& 1295348\\
RG\_randomB65 &
2	& 500	& 81172	& --& 1210096
								& 1301186
										& 1300226	& 1299860
															& 1289980
																	& 1299550
																			& 1300660& {\bf 1282456}& 1298326& 1286966& 1293572\\
RG\_randomA75 &
2	& 500	& 93347	& --& 1429246
								& 1496336
										& 1495498	& 1495320
															& 1486398
																	& 1495054
																			& 1495980& {\bf 1479758}& 1493802& 1484062& 1490928\\
RG\_randomB75 &
2	& 500	& 93399	& --& 1430182
								& 1497266
										& 1496448	& 1495956
															& 1487202
																	& 1496172
																			& 1497070& {\bf 1480906}& 1494338& 1484748& 1492458\\
RG\_randomA85 &
2	& 500	& 106047& --& 1657846
								& 1699742
										& 1699524	& 1699104
															& 1692306
																	& 1698948
																			& 1699578& {\bf 1687470}& 1698310& 1690196& 1693914\\
RG\_randomB85 &
2	& 500	& 106111& --& 1658998
								& 1700626
										& 1700554	& 1700062
															& 1692992
																	& 1699822
																			& 1700268& {\bf 1688546}& 1698760& 1691352& 1697034\\
RG\_randomA95 &
2	& 500	& 118499& --& 1881982
								& 1899982
										& 1899538	& 1899100
															& 1895412
																	& 1899376
																			& 1899696& {\bf 1893354}& 1898482& 1894696& 1897084\\
RG\_randomB95 &
2	& 500	& 118606& --& 1883908
								& 1900908
										& 1901264	& 1900678
															& 1897246
																	& 1900698
																		 	& 1900804& {\bf 1895088}& 1900376& 1896222& 1898688\\
			\hline
RG\_randomA5 &
7	& 500	& 6126	& --& 21504	& 40774	& 46738	& 40522	& 38070	& 39990	& 40518	& {\bf 36392}	& 39344	& 37494 & 39780\\
RG\_randomB5 &
7	& 500	& 6175	& --& 21700	& 41204	& 47124	& 40886	& 38358	& 40222	& 40816	& {\bf 36570}	& 39498	& 37794	& 40504\\
RG\_randomA15 &
7	& 500	& 18654	& --& 84924	& 124204& 142714& 123766& 119752& 123254& 124042& {\bf 117348}	& 122284& 118986& 124010\\
RG\_randomB15 &
7	& 500	& 18887	& --& 86322	& 125500& 144464& 125316& 121384& 124626& 125386& {\bf 118666}	& 123678& 120570& 125044\\
RG\_randomA25 &
7	& 500	& 31254	& --& 160524& 207686& 239158& 207686& 203046& 206856& 207712& {\bf 199806}	& 205834& 201802& 207458\\
RG\_randomB25 &
7	& 500	& 31114	& --& 159684& 206620& 238134& 206588& 201778& 205850& 206552& {\bf 198944}	& 204876& 201050& 206506\\
RG\_randomA35 &
7	& 500	& 43605	& --& 234630& 289994& 333908& 289584& 284438& 288968& 289716& {\bf 281458}	& 287540& 283704& 289994\\
RG\_randomB35 &
7	& 500	& 43595	& --& 234570& 290000& 333834& 289650& 284588& 288862& 289652& {\bf 281502}	& 287806& 283608& 290000\\
RG\_randomA45 &
7	& 500	& 56653	& --& 312918& 376680& 433734& 376316& 371212& 375824& 376474& {\bf 368306}	& 374474& 370416& 376604\\
RG\_randomB45 &
7	& 500	& 55627	& --& 306762& 369694& 426026& 369460& 364486& 368858& 369702& {\bf 361318}	& 367498& 363534& 369646\\
RG\_randomA55 &
7	& 500	& 68320	& --& 382920& 454306& 523376& 454142& 448710& 453224& 453876& {\bf 446038}	& 452220& 448168& 454306\\
RG\_randomB55 &
7	& 500	& 68701	& --& 385206& 456812& 526276& 456536& 451350& 455894& 456454& {\bf 448314}	& 454596& 450618& 456812\\
RG\_randomA65 &
7	& 500	& 81279	& --& 460795& 540234& 622664& 540098& 535146& 539482& 540146& {\bf 532718}	& 538368& 534564& 540234\\
RG\_randomB65 &
7	& 500	& 81172	& --& 460159& 539620& 621798& 539254& 534380& 538760& 539478& {\bf 532046}	& 538248& 534046& 539620\\
RG\_randomA75 &
7	& 500	& 93347	& --& 548776& 620442& 715076& 620622& 616240& 619842& 620570& {\bf 613984}	& 619312& 615588& 620442\\
RG\_randomB75 &
7	& 500	& 93399	& --& 549192& 621312& 715644& 620878& 616030& 620392& 620938& {\bf 614182}	& 619242& 616400& 621312\\
RG\_randomA85 &
7	& 500	& 106047& --& 650376& 705130& 812572& 705052& 701458& 704718& 704864& {\bf 699976}	& 703688& 701318& 705130\\
RG\_randomB85 &
7	& 500	& 106111& --& 650888& 705676& 812948& 705388& 701914& 705162& 705472& {\bf 700330}	& 704000& 701820& 705286\\
RG\_randomA95 &
7	& 500	& 118499& --& 749992& 788218& 907730& 787964& 786042& 787888& 788060& {\bf 785158}	& 787360& 785984& 788214\\
RG\_randomB95 &
7	& 500	& 118606& --& 750848& 788782& 908742& 788638& 786674& 788496& 788652& {\bf 785996}	& 788036& 786804& 788642\\
			\hline\hline
Pet03\_randomA1 &
2	& 1000	& 4974	& --& 29154	& 89750	& 88988	& 88944	& 80096	& 87408	& 86806	& {\bf 71874}	& 81264	& 75440	& 86824\\
Pet03\_randomA2 &
2	& 1000	& 24738	& --& 239917& 446298& 444500& 444384& 426856& 442944& 445890& {\bf 411242}	& 438108& 42540	& 443786\\
Pet03\_randomA3 &
2	& 1000	& 49820	& --& 577482& 897992& 895654& 895760& 873202& 894196& 897554& {\bf 852854}	& 		& 864912& 894160\\
Pet03\_randomA4 &
2	& 1000	& 8177	& --& 56759 & 147424& 146646& 146528& 135366& 145032& 146066& {\bf 125374}	&		& 130530& 144756\\
Pet03\_randomG4 &
2	& 1000	& 8173	& --& 56961	& 147164& 146030& 146536& 98482	& 93990	& 96648	& {\bf 74282}	&		& 106720& 135804\\
Pet03\_hc10 &
2	& 1024	& 5120	& --& 29696	& {\bf 56320}
										& 91468	& 91672	& {\bf 56320}
																& 88684	& 84266	& {\bf 56320}	& {\bf 56320}
																										& {\bf 56320}
																												& {\bf 56320}\\
Pet03\_mesh33x33 &
2	& 1089	& 2112	& --& 8320	& 18942	& 41788	& 38714	& 19350	& 26152	& 23268	& {\bf 18722}	& 		& 18900	& 18904\\
Pet03\_3elt &
2	& 4720	& 13722	& --& 63462	& 169346& 328306& 313178& 189096& 174530& 220584& {\bf 120332}	& --	& 132904& 168708\\
Pet03\_airfoil1 &
2	& 4253	& 12289	& --& 56732	& 148896& 293980& 273364& 165388& 153078& 188894& {\bf 106416}	& --	& 119024& 148382\\
{\color{grey} Pet03\_crack} &
{\color{grey} 2}
	& {\color{grey} 10240}
		& {\color{grey} 30380}
			& {\color{grey} --}
				& {\color{grey} 145618}
					& {\color{grey} 726664}
						& {\color{grey} 788288}
							& {\color{grey} 762606}
								& {\color{grey} 426592}
									& {\color{grey} 393516}
										& {\color{grey} 497116}
											& {\color{grey} --}
												& {\color{grey} --}
													& {\color{grey} 441042}
														& {\color{grey} 696122}\\
Pet03\_whitaker3 &
2	& 9800	& 28989	& --& 134741& 375730& 752132& 723426& 371946& 355240& 492390& {\bf 301320}	& --	& 335630& 375298\\
Pet03\_big &
2	& 15606	& 45878	& --& 212875& 650814& 1190894
												& 1189886
														& 726976& 650746& 830690& {\bf 436908}	& --	& 482936& 649748\\
{\color{grey} Pet03\_wave} &
{\color{grey} 2}
	& {\color{grey} 156317}
		& {\color{grey} 1059331}
			& {\color{grey} --}
				& {\color{grey} 6884189}
					& {\color{grey} 21067766}
						& {\color{grey} 36008138}
							& {\color{grey} 34921840}
								& {\color{grey} 23977688}
									& {\color{grey} 21016364}
										& {\color{grey} 23484964}
											& {\color{grey} --}
												& {\color{grey} --}
													& {\color{grey} --}
														& {\color{grey} \bf 20711426}\\
Pet03\_c1y &
2	& 828	& 1749	& --& 8609	& 24712	& 31192	& 30752	& 21392	& 26068	& 22924	& {\bf 16884}	& 21454	& 19846	& 23174\\
Pet03\_c2y &
2	& 980	& 2102	& --& 10246	& 29726	& 37394	& 37392	& 25282	& 30232	& 26722	& {\bf 20478}	& 26246	& 24110	& 11592\\
Pet03\_c3y &
2	& 1327	& 2844	& --& 13578	& 41996	& 56492	& 54426	& 35066	& 44664	& 38692	& {\bf 28810}	&		& 33736	& 33736\\
Pet03\_c4y &
2	& 1366	& 2915	& --& 13529	& 43490	& 57848	& 56106	& 37034	& 44186	& 38306	& {\bf 27930}	&		& 34124	& 42814\\
Pet03\_c5y &
2	& 1202	& 2577	& --& 12120	& 37636	& 50712	& 48414	& 32894	& 37942	& 33524	& {\bf 25572}	&		& 29328	& 35858\\
Pet03\_gd95c &
2	& 62	& 144	& --& 643	& 1016	& 1384	& 1354	& 1080	& 1024	& 1002	& {\bf 866}		& 1044	& 916	& 920\\
Pet03\_gd96a &
2	& 1096	& 1676	& --& 7021	& 30310	& 33124	& 30774	& 23050	& 27908	& 19060	& {\bf 18004}	&		& 19926	& 28526\\
Pet03\_gd96b &
2	& 111	& 193	& --& 971	& 2410	& 2246	& 2200	& 1762	& 1820	& 1768	& {\bf 1486}	& 1636	& 1760	& 1576\\
Pet03\_gd96c &
2	& 65	& 125	& --& 495 	& 1276	& 1394	& 1236	& 882	& 936	& 964	& {\bf 824}		& 972	& 950	& 844\\
Pet03\_gd96d &
2	& 180	& 228	& --& 1002	& 2446	& 3070	& 2952	& 2592	& 2802	& 2050	& {\bf 1822}	& 2038	& 2054	& 2024\\
			\hline
Pet03\_randomA1 &
7	& 1000	& 4974	& --& 14006	& 35988	& 37952	& 35680	& 32204	& 35032	& 35120	& {\bf 29574}	& 33980	& 30608	& 35002\\
Pet03\_randomA2 &
7	& 1000	& 24738	& --& 96266	& 178872& 189270& 178334& 171354& 177498& 178924& {\bf 166598}	& 177006& 169020& 178052\\
Pet03\_randomA3 &
7	& 1000	& 49820	& --& 244920& 360072& 381548& 359372& 350714& 358698& 359846& {\bf 343664}	& 357742& 347592& 359368\\
Pet03\_randomA4 &
7	& 1000	& 8177	& --& 26710	& 59134	& 62466	& 58842	& 54702	& 58128	& 58810	& {\bf 51290}	& 57198	& 52990	& 58048\\
Pet03\_randomG4 &
7	& 1000	& 8173	& --& 26697	& 59124	& 62296	& 58840	& 40294	& 39452	& 40402	& {\bf 31838}	& 46582	& 43626	& 56006\\
Pet03\_hc10 &
7	& 1024	& 5120	& --& 14336	& 26204	& 39034	& 36824	& 26280	& 35020	& 34452	& {\bf 25892}	& 27580	& 25940	& 26192\\
Pet03\_mesh33x33 &
7	& 1089	& 2112	& --& 4224	& 8294	& 16044	& 15298	& 8326	& 10478	& 10076	& {\bf 8224}	& 9172	& 8262	& 8286\\
Pet03\_bintree10 &
7	& 1023	& 1022	& --& 2044	& 7384	& 7752	& 7288	& 7384	& 6488	& {\bf 2856}
																				& 3030			& 3096	& 4418	& 6436\\
Pet03\_3elt &
7	& 4720	& 13722	& --& 27694	& 68636	& 132340& 120844& 73754	& 71042	& 86290	& {\bf 52420}	& --	& 55328	& 68576\\
Pet03\_airfoil1 &
7	& 4253	& 12289	& --& 24792	& 60750	& 118422& 107822& 64224	& 63170	& 78482	& {\bf 46768}	& --	& 49868	& 60466\\
Pet03\_crack &
7	& 10240	& 30380	& --& 69741	& 277928& 293238& 287434& 168802& 158460& 193856& {\bf 138178}	& --	& 175252& 269060\\
Pet03\_whitaker3 &
7	& 9800	& 28989	& --& 58482	& 154188& 279798& 273056& 153490& 147924& 191430& {\bf 126198}	& --	& 137200& 153272\\
Pet03\_big &
7	& 15606	& 45878	& --& 92511	& 261418& 442852& 442412& 286180& 264862& 317434& {\bf 186158}	& --	& 199324& 260680\\
{\color{grey} Pet03\_wave} &
{\color{grey} 7}
	& {\color{grey} 156317}
	 	& {\color{grey} 1059331}
			& {\color{grey} --}
				& {\color{grey} 3299657}
					& {\color{grey} 8223278}
						& {\color{grey} 14474822}
							& {\color{grey} 13238898}
								& {\color{grey} 9203558}
									& {\color{grey} 8150052}
										& {\color{grey} 9050366}
											& {\color{grey} --}
												& {\color{grey} --}
													&  {\color{grey} --}
														& {\color{grey} \bf 8110512}\\
Pet03\_c1y &
7	& 828	& 1749	& --& 4075	& 10224	& 13294	& 12326	& 8778	& 10426	& 9454	& {\bf 7508}	& 8834	& 8290	& 9652\\
Pet03\_c2y &
7	& 980	& 2102	& --& 4822	& 12218	& 16010	& 15044	& 10540	& 12132	& 11058	& {\bf 8724}	& 11174	& 9986	& 11592\\
Pet03\_c3y &
7	& 1327	& 2844	& --& 6436	& 16810	& 21678	& 20930	& 14652	& 17786	& 15230	& {\bf 12216}	& 16430	& 13676	& 16612\\
Pet03\_c4y &
7	& 1366	& 2915	& --& 6442	& 17168	& 22212	& 21438	& 14884	& 17320	& 15490	& {\bf 12460}	& 16434	& 13732	& 16872\\
Pet03\_c5y &
7	& 1202	& 2577	& --& 5750	& 15112	& 19450	& 18730	& 13452	& 15170	& 13618	& {\bf 10946}	&		& 11978	& 14520\\
Pet03\_gd95c &
7	& 62	& 144	& --& 320	& 474	& 780	& 608	& 446	& 460	& 492		& 410		& 500	& {\bf 404}
																												& 450\\
Pet03\_gd96a &
7	& 1096	& 1676	& --& 3800	& 12056	& 12736	& 12092	& 9514	& 11106	& 7990	& {\bf 7768}	& 		& 8308	& 11050\\
Pet03\_gd96b &
7	& 111	& 193	& --& 491	& 1030	& 1062	& 928	& 744	& 716	& 772		& {\bf 682}	& 726	& 788	& 696\\
Pet03\_gd96c &
7	& 65	& 125	& --& 250	& 574	& 680	& 542	& 426	& 448	& 448		& {\bf 378}	& 422	& 438	& 392\\
Pet03\_gd96d &
7	& 180	& 228	& --& 555	& 1028	& 1258	& 1188	& 1060	& 1126	& 876		& {\bf 838}	& 1006	& 886	& 958\\
			\hline
			\multicolumn{16}{c}{\parbox{\LTcapwidth}{}}\\
			\caption{\em Summary for the instances without a known optimal solution.}
			\label{tab:summaryForTheInstancesWithoutAKnownOptimum}
		\end{longtable}
	\end{center}

	\paragraph{List of acronyms}
		\begin{itemize}
			\item OS = optimal solution (if known).
			\item DB = degree bound.
			\item NAM = normal arrangement. The vertices $\{v_1,v_2,\ldots, v_n\}$ of the guest graph are mapped to the leaves of the $d$-regular
 tree in their canonical ordering, i.e.\  by $\phi(v_i)= b_i$, for $i=1,2,\ldots,n$. 
			\item RAM = random arrangement for $k = 1000$.
$k$ random mappings of the  vertices  of the guest graph  into the leaves of the $d$-regular
 tree are constructed, their objective function values are computed, and the random mapping with the best objective function value is selected.
			\item RCAM = random contiguous arrangement  for $k = 1000$. $k$ random contiguous mappings of the  vertices  of the guest graph  into the leaves of the $d$-regular
 tree are constructed, their objective function values are computed, and the random mapping with the best objective function value is selected.
			\item G2 = arrangement produced by the leaf-driven greedy heuristic, see~Section~\ref{sub:greedy}.
			\item BFSG = arrangement produced by the breadth-first search  based greedy heuristics which tries each vertex as the starting vertex, 
see~Section~\ref{sub:greedy}. If the graph has more then one connected components, they are arranged in a random order. 
			\item TFSG = arrangement produced by the depth-first search  based greedy heuristics which tries each vertex as the starting vertex, 
see~Section~\ref{sub:greedy}.If the graph has more then one connected components, they are arranged in a random order. 
			\item CHLS = arrangement produced by the construction heuristic which uses the local search approach to solve the MCBSSP, 
see~Section~\ref{sub:constrDescr}.
			\item PEHVNA = arrangement produced by the pair-exchange heuristic for vertices which starts with the normal arrangement, 
see~Section~\ref{subsub:PE}.
			\item SFHWI = arrangement produced by the shift-flip heuristic which accepts non-improving shifts, see~Section~\ref{subsub:SL}. 
The algorithm terminates  if  no improvement is reached after $3$ days of running time.
			\item -- = the solution could not be found in a reasonable amount of time.
		\end{itemize}
	
\end{landscape}


\restoregeometry

\normalsize


\end{document}